\documentclass[12pt, reqno]{amsart}
\usepackage{amsmath, amsthm, amscd, amsfonts, amssymb, graphicx, xcolor}
\usepackage[bookmarksnumbered, colorlinks, plainpages]{hyperref}
\usepackage{amsmath}
\usepackage{cases}
\usepackage{amssymb}
\usepackage{mathrsfs}
\usepackage{euscript}
\usepackage{dsfont}
\usepackage{graphicx}
\usepackage{float}
\usepackage{graphicx}
\newcommand*\extsim[1][3.0]{\mathop{\scalebox{#1}[1.0]{\hbox{\textasciitilde}}}}

\allowdisplaybreaks[4]
\newtheorem{theorem}{Theorem}[section]
\newtheorem{lemma}{Lemma}[section]

\theoremstyle{definition}

\newtheorem{conjecture}[theorem]{Conjecture}

\theoremstyle{remark}
\newtheorem{remark}[theorem]{Remark}
\numberwithin{equation}{section}

\begin{document}
	\setcounter{page}{1}
	

	
	\title[]{An improvement on the largest prime factors of consecutive integers}
	
	\author[]{Zhiyuan Yang}
	
	\address{School of Mathematics,  Shandong University, Jinan 250100, Shandong, China}
	\email{\tt zhiyuan.yang@mail.sdu.edu.cn}
	
	
	

	
	
	\begin{abstract}
		Let $P^+(n)$ denote the largest prime factor of $n$. One of Erd\H{o}s and Turán’s
		conjectures asserts that the asymptotic density of integers $n$ satisfying $P^+(n)<P^+(n+1)$ is 1/2. In this paper, we prove that this density is larger than 0.280,
		which improves the previous result 0.2017 by Lü and Wang (2025). We also prove that there exists
		a positive density of $n$ such that $P^+(n)<P^+(n+1)<x^{41/107+\varepsilon}$. Define $T_c(x):=\#\{p\leq x:P^+(p-1)\geq p^c\}$.
		For $1/2<c<1$, we also show that\begin{align*}
			\mathop{\lim \sup}_{x\rightarrow\infty}\frac{T_c(x)}{\pi(x)}\leq \min\left(-\frac{7}{2}\log c,\frac{1-\delta}{2c}\right),
		\end{align*}
		where $\delta=\delta(c)>0$.
		\newline
		\newline
		\noindent \textit{Keywords.} Erd\H{o}s-Turán’s conjecture
		\newline
		\noindent 
	\end{abstract} \maketitle
	
	
	\section{Introduction}
	Let $P^+(n)$ denote the largest prime factor of $n$ with the convention that $P^+(1)=1$. In the 1930s, Erd\H{o}s and Turán [22] formulated the following conjecture in a correspondence.
	\begin{conjecture}
		(Erd\H{o}s-Turán) For $x\rightarrow\infty$, we have
		\begin{align*}
			\#\{n\leq x:P^+(n)<P^+(n+1)\}\sim \frac{1}{2}x.
		\end{align*}
	\end{conjecture}
	Conjecture 1.1 was of constant interest for Erd\H{o}s even though he thought it might be intractable by any technique at our disposal (see [9] and [24]). Later in 1978, Erd\H{o}s and Pomerance [10] conjectured that the largest prime factors of $n$ and $n+1$ are ``independent random variables".
	\begin{conjecture}
		(Erd\H{o}s-Pomerance) For any $a,b\in[0,1]$, denote by $B(x;a,b)$ the number of $n\leq x$ with $P^+(n)\leq x^a$ and $P^+(n)\leq x^b$,
		\begin{align*}
			B(x;a,b):=\#\{n\leq x:P^+(n)\leq x^a,P^+(n+1)\leq x^b\}.
		\end{align*}
		Then we have
		\begin{align*}
			\lim_{x\rightarrow\infty}x^{-1}B(x;a,b)=\rho\left(\frac{1}{a}\right)\rho\left(\frac{1}{b}\right),
		\end{align*} 
		where $\rho(t)$ is the Dickman-de Bruijn function.
	\end{conjecture}
	We remark that Conjecture 1.2 implies Conjecture 1.1 (see [25]). In 2011, De Koninck and Doyon [3] formulated a more general conjecture.
	\begin{conjecture}
		(De Koninck-Doyon) Fix an aribitrary integer $k\geq 2$ and let $n$ be a large number. Let $a_1,a_2,\cdots,a_k$ be any permutation of the numbers $0,1,\cdots,k-1$. Then we have\begin{align*}
			\#\{n\leq x:P^+(n+a_1)<P^+(n+a_2)<\cdots<P^+(n+a_k)\}\sim\frac{1}{k!}x.
		\end{align*}
	\end{conjecture}
	In this paper, we mainly focus on Conjecture 1.1. There has been some progress towards these problems. In 1978, Erd\H{o}s and Pomerance [10] first proved that there exists a positive asymptotic density of integers $n$ such that $P^+(n)<P^+(n+1)$. In fact, they showed that\begin{align*}
		\#\{n\leq x:P^+(n)<P^+(n+1)\}>0.0099x.
	\end{align*}
	In 2005, the asymptotic density was improved to 0.05544 by La Bretèche, Pomerance and Tenenbaum [4], and to 0.05866 by Fouvry's arguments in ``Further remarks" of the same paper. Later, the constant 0.05866 was improved successively to 0.1063 and 0.1356 by Wang in [26, 28]. The current record is 0.2017 obtained by Lü and Wang in [17].
	
	On the other hand, in 2001 Rivat in [21] proved a $P^+_y(n)$-version of the Erd\H{o}s-Turán conjecture for some small $y$, where $P^+_y(n)$ denotes the greatest prime factor $p$ of $n$ satisfying $p\leq y$. In fact, Rivat showed that for $3\leq y\leq \exp(\log x/(100\log\log x))$, we have
	\begin{align*}
		\#\{n\leq x:P^+_y(n)<P^+_y(n+1)\}\sim\frac{1}{2}x.
	\end{align*}
	In 2018, Ter\"{a}v\"{a}inen [25] proved a logarithmic version of Conjecture 1.1:
	\begin{align*}
		\delta(\{n\in\mathbb{N}:P^+(n)<P^+(n+1)\})=\frac{1}{2},
	\end{align*}
	where $\delta$ is defined by
	\begin{align*}
		\delta(A)=\lim_{x\rightarrow\infty}\frac{1}{\log x}\sum_{n\leq x,n\in A}\frac{1}{n}
	\end{align*}
	whenever it exists. In 2019, Tao and Ter\"{a}v\"{a}inen [23] showed that there is an exceptional set $\mathscr{X}$ of logarithmic density 0 such that\begin{align*}
		\lim_{x\rightarrow\infty,x\notin\mathscr{X}}\frac{1}{x}\sum_{\substack{n\leq x\\P^+(n)<P^+(n+1)}}1=\frac{1}{2}.
	\end{align*}
	In 2021, Wang [29] showed that Conjecture 1.1 holds under the Elliott-Halberstam conjecture for friable integers. Recently, Jiang, L\"u and Wang [15] proved that Conjecture 1.1 with the pattern $P^+(n)<P^+(n+h)$ holds on average, i.e. for almost of shifts $h$ with $(\log x)^{1+\varepsilon}\leq h\leq x^{1-\varepsilon}$.
	
	In this paper, for $P^+(n)<P^+(n+1)<x^{1-c}$, we have optimized the calculation and advance the idea of L\"u and Wang in several respects (see \S 3.2). We prove that the asymptotic density is larger than 0.280. 
	\begin{theorem}
		For $x\rightarrow\infty$, we have
		\begin{align*}
			\#\{n<x:P^+(n)<P^+(n+1)\}>0.280x.
		\end{align*}
		The lower bound is also true for the pattern $P^+(n)>P^+(n+1)$.
	\end{theorem}
	
	Suppose that $n$ and $n+1$ are two $y$-smooth numbers, where $y=x^{1/u}$. We may guess that
	\begin{align*}
		\#\{n\leq x:P^+(n)<P^+(n+1)<y\}\gg x.
	\end{align*}
	In fact, we can prove that the asymptotic density is $\rho^2(u)/2$ under the Elliott-Halberstam
	conjecture for friable integers by following the argument
	of Theorem 2 in [29].
	We also can get $\#\{n\leq x:P^+(n)<P^+(n+1)<x^{1/2+\varepsilon}\}\gg x$ by a simple derivation (see [28, p.374-375]). In this paper, we obtain the following result.
	
	\begin{theorem}
		For $x\rightarrow\infty$, we have
		\begin{align*}
			\#\{n<x:P^+(n)<P^+(n+1)<x^{41/107+\varepsilon}\}\gg x.
		\end{align*}
		The lower bound is also true for the pattern $P^+(n+1)<P^+(n)<x^{41/107+\varepsilon}$.
	\end{theorem}
	
	Define
	\begin{align*}
		T_c(x):=\#\{p\leq x:P^+(p-1)\geq p^c\}.
	\end{align*}
	We are interested in the upper bound of $T_c(x)$. By the proof of a former result of Erd\H{o}s [8, p.212--213], one could conclude that
	\begin{align*}
		\mathop{\lim\sup}_{x\rightarrow \infty}\frac{T_c(x)}{\pi(x)}\rightarrow0, \ \text{as}\ c\rightarrow 1.
	\end{align*}
	In [2], Chen and Chen stated a conjecture that implies in particular that for $1/2\leq c<1$, one has
	\begin{align*}
		\mathop{\lim\inf}_{x\rightarrow \infty}\frac{T_c(x)}{\pi(x)}\geq \frac{1}{2}.
	\end{align*}
	In 2023,
	Ding [5] disproved this conjecture by showing the existence of $\delta>0$ such that\begin{align*}
		\mathop{\lim\sup}_{x\rightarrow\infty}\frac{T_c(x)}{\pi(x)}<\frac{1-2\delta}{2c}
	\end{align*}
	for $3/4<c<1$.
	Subsequently, Ding [6] obtained a quantitative form of Erd\H{o}s’
	result:\begin{align*}
		\mathop{\lim\sup}_{x\rightarrow\infty}\frac{T_c(x)}{\pi(x)}\leq 8(c^{-1}-1).
	\end{align*}
	Very recently, Ding and Wang [7] improved Ding’s upper bound to
	\begin{align*}
		\mathop{\lim\sup}_{x\rightarrow\infty}\frac{T_c(x)}{\pi(x)}\leq -\frac{7}{2}\log c.
	\end{align*}
	In this paper, the method used to prove Theorem 1.4 can also be applied to the upper bound of $T_c(x)$. In fact, we show that
	\begin{theorem}
		For $1/2<c<1$, we have\begin{align*}
			\mathop{\lim \sup}_{x\rightarrow\infty}\frac{T_c(x)}{\pi(x)}\leq \min\left(-\frac{7}{2}\log c,\frac{1-\delta}{2c}\right),
		\end{align*}
		where $\delta=\delta(c)>0$ given by
		\begin{figure}[ht]
			\centering
			\includegraphics[width=0.9\textwidth]{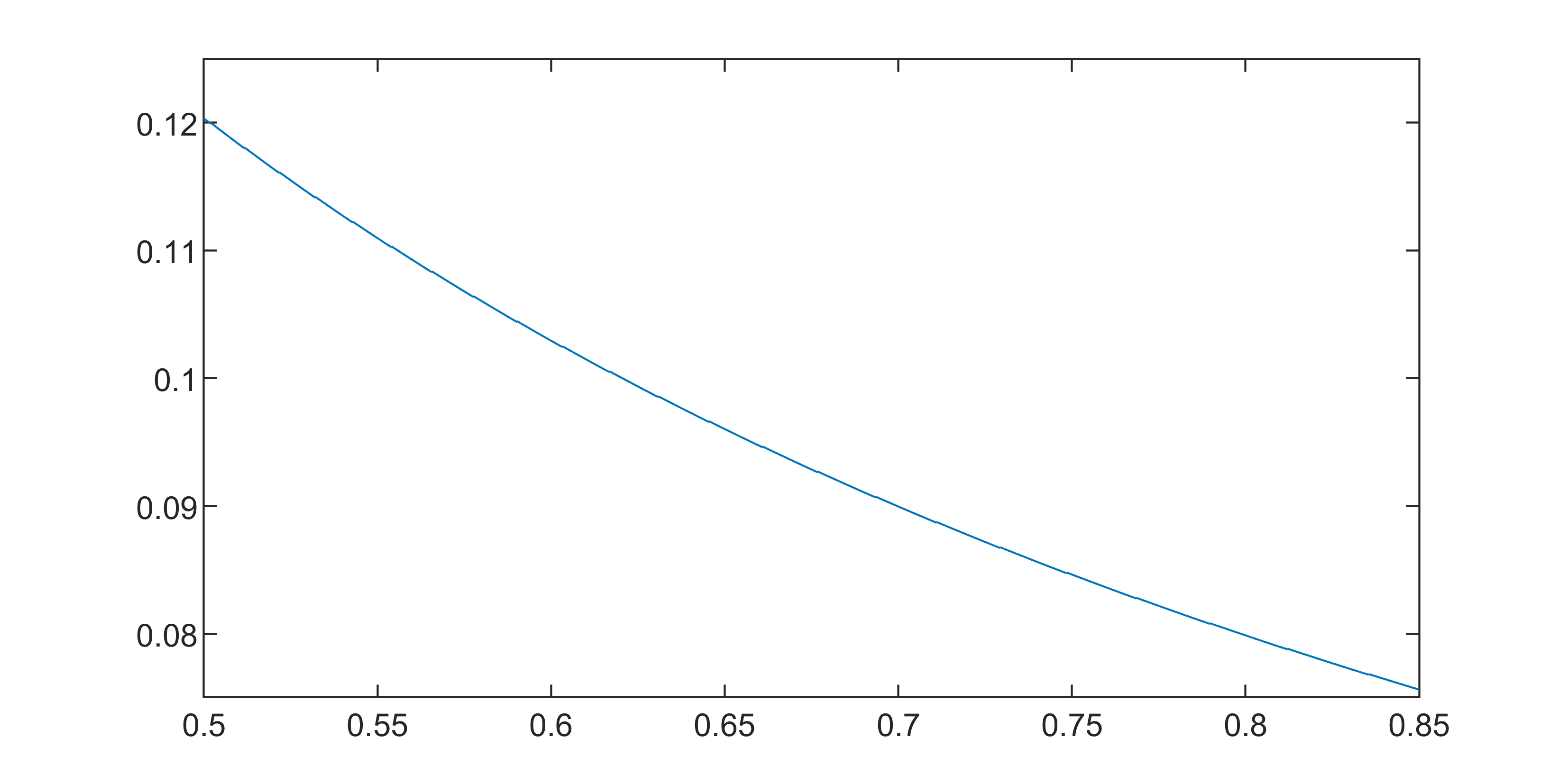} 
			\caption{$c\rightarrow\delta(c)$}
		\end{figure}\\
		In particular, we can obtain a better value for $\delta$ by following the argument in Remark 6.2.
		
	\end{theorem}

	\section{Lemmas}
	\begin{lemma}
		For $\varepsilon>0$, we have
		\begin{align*}
			\Psi(x,y)=\sum_{n\leq x,P^+(n)\leq y}1=x\rho(u)\left(1+O_\varepsilon\left(\frac{\log(u+1)}{\log y}\right)\right)
		\end{align*}
		uniformly for
		\begin{align*}
			x\geq x_0(\varepsilon),\qquad \exp{(\log_2x)^{5/3+\varepsilon}}\leq y\leq x,
		\end{align*}
		where $u=\log x/\log y$ and $\rho(u)$ is the Dickman-de Bruijn function which is defined by differential equation
		\begin{equation}
			\begin{aligned}
				&\left\{ \begin{aligned}
					&\rho(u)=1, \quad&& 0\leq u\leq 1,\\&u\rho'(u)=-\rho(u-1)\quad&& u>1.
				\end{aligned}\right.\nonumber
			\end{aligned}
		\end{equation}
	\end{lemma}
	\begin{proof}
		This is [13, Theorem 1].
	\end{proof}
	
	\begin{lemma}
		For any given positive constant $A>0$, there exists a constant $B=B(A)>0$ such that the estimate
		\begin{align*}
			\sum_{q\leq x^{1/2}/(\log x)^B}\max_{t\leq x}\max_{(a,q)=1}\Big|\sum_{\substack{n\leq t,P^+(n)\leq y\\n\equiv a\mkern-15mu\pmod{q}}}1-\frac{1}{\varphi(q)}\sum_{\substack{n\leq t,P^+(n)\leq y\\(n,q)=1}}1\Big|\ll\frac{x}{(\log x)^A}
		\end{align*}
		holds uniformly for $2\leq y\leq x$.
	\end{lemma}
	\begin{proof}
		This lemma can be found in [30] and [12, Theorem 6]. In [30], Wolke proved a theorem of Bombieri-Vinogradov type for sifted numbers without small prime factors and announced that the dual case for friable integers can be treated analogously (see [30, p.146]).
	\end{proof}
	
	\begin{lemma}
		Let $b\in\mathbb{Z}^*$ and $A$, $\varepsilon>0$. Then there exists $C=C(A,\varepsilon)>0$ such that, in the range $x>2$, $(\log x)^C\leq y\leq x^{1/C}$, one has
		\begin{align*}
			\sum_{\substack{q\leq x^{66/107-\varepsilon}\\(q,b)=1}}\Big|\sum_{\substack{n\leq x,P^+(n)\leq y\\n\equiv b\mkern-15mu\pmod{q}}}1-\frac{1}{\varphi(q)}\sum_{\substack{n\leq x,P^+(n)\leq y\\(n,q)=1}}1\Big|\ll_{b,A,\varepsilon}\frac{\Psi(x,y)}{(\log x)^A}.
		\end{align*}
	\end{lemma}
	\begin{proof}
		This lemma is [19, Theorem 1.1]. We can also see the works of Pascadi in [20].
	\end{proof}
	
	\begin{lemma}
		Let $D\geq 2$ and $L>1$. Let $\mathscr{P}$ denote a set of primes. Let $z\geq 2$ and write $P(z):=\prod_{p\leq z,p\in\mathscr{P}}p$. There exist two sequences $\{\lambda_d^\pm\}_{d=1}^{\infty}$ of real numbers, vanishing for $d>D$ or $\mu(d)=0$, satisfying $\lambda_1^\pm=1$, $|\lambda_d^{\pm}|\leq 1$, $\lambda^-*1\leq \mu*1\leq \lambda^+*1$, and such that
		\begin{align*}
			&\sum_{d|P(z)}\lambda_d^+\frac{\omega(d)}{d}\leq \prod_{\substack{p\leq z\\p\in \mathscr{P}}}\left(1-\frac{\omega(p)}{p}\right)\left(F(s)+O\left(\frac{e^{L^{1/2}-s}}{(\log D)^{1/3}}\right)\right),\\&\sum_{d|P(z)}\lambda_d^-\frac{\omega(d)}{d}\geq \prod_{\substack{p\leq z\\p\in \mathscr{P}}}\left(1-\frac{\omega(p)}{p}\right)\left(f(s)+O\left(\frac{e^{L^{1/2}-s}}{(\log D)^{1/3}}\right)\right)
		\end{align*}
		uniformly for all multiplicative function $\omega$ satisfying
		\begin{align*}
			&(i)\ 0<\omega(p)<p\ (p\in\mathscr{P}),\\&(ii)\ \prod_{\substack{u<p\leq v\\p\in\mathscr{P}}}\left(1-\frac{\omega(p)}{p}\right)^{-1}\leq \frac{\log v}{\log u}\left(1+\frac{L}{\log u}\right)\ (2\leq u\leq v\leq z),
		\end{align*}
		where $s=\log D/\log z$ and $f(s)$, $F(s)$ are determined by the differential-difference equations
		\begin{equation}
			\begin{aligned}
				&\left\{ \begin{aligned}
					&sF(s)=2e^{\gamma}, \quad&& 1\leq s\leq 2;\\&sf(s)=0,\quad&& 0\leq s\leq 2;\\&(sF(s))'=f(s-1),(sf(s))'=F(s-1),\quad&&s\geq 2.
				\end{aligned}\right.\nonumber
			\end{aligned}
		\end{equation}
	\end{lemma}
	\begin{proof}
		This lemma is the Rosser-Iwaniec sieve [14].
	\end{proof}  
	\begin{lemma}
		For any given positive constant $A>0$, there exists a constant $B=B(A)>0$ such that the estimate\begin{align*}
			\sum_{q\leq x^{1/2}/(\log x)^B}\max_{y\leq x}\max_{(a,q)=1}\Big|\sum_{\substack{L_1<l\leq L_2\\(l,q)=1}}f(l)\Big(\sum_{\substack{lp\leq y\\lp\equiv a\mkern-15mu\pmod{q}}}1-\frac{\mathrm{li}(y/l)}{\varphi(q)}\Big)\Big|\ll\frac{x}{(\log x)^A}
		\end{align*}
		and
		\begin{align*}
			\sum_{q\leq x^{1/2}/(\log x)^B}\max_{y\leq x}\max_{(a,q)=1}\Big|\sum_{\substack{L_1<l\leq L_2\\(l,q)=1}}f(l)\Big(\sum_{\substack{p\leq y/L_2\\lp\equiv a\mkern-15mu\pmod{q}}}1-\frac{\mathrm{li}(y/L_2)}{\varphi(q)}\Big)\Big|\ll\frac{x}{(\log x)^A}
		\end{align*}
		hold for $(\log y)^{2B}<L_1\leq L_2<x^{1-\varepsilon}$, where $|f(l)|\leq 1$ and the constant implied by the symbol ``$\ll$" depends only on $\varepsilon$ and $A$.
	\end{lemma}
	\begin{proof}
		This is [17, Lemma 2.3] (see [18, Theorem 2]). 
	\end{proof}
	
	Define
	\begin{align*}
		\pi(x;l,a,q)=\sum_{\substack{lp\leq x\\lp\equiv a\mkern-15mu\pmod{q}}}1.
	\end{align*}
	
	\begin{lemma}
		Let $a\neq 0$ be a given integer, and let $A>0$ and $\varepsilon>0$. For any well factorable
		function $\lambda(q)$ of level $Q$, the following estimates
		\begin{align*}
			\sum_{\substack{q\leq Q\\(a,q)=1}}\lambda(q)\sum_{\substack{L_1\leq l\leq L_2\\(l,q)=1}}\left(\pi(x;l,a,q)-\frac{\pi(x/l)}{\varphi(q)}\right)\ll\frac{x}{(\log x)^A}
		\end{align*}
		and \begin{align*}
			\sum_{\substack{q\leq Q\\(a,q)=1}}\lambda(q)\sum_{\substack{L_1\leq l\leq L_2\\(l,q)=1,2|l}}\left(\pi(x;l,a,q)-\frac{\pi(x/l)}{\varphi(q)}\right)\ll\frac{x}{(\log x)^A}
		\end{align*}
		hold for $Q\leq x^{4/7-\varepsilon}$ and $1\leq L_1\leq L_2\leq x^{1-\varepsilon}$, where the implied constants depend only on $a$, $A$ and $\varepsilon$, $\varphi(n)$ is the Euler totient function.
	\end{lemma}
	\begin{proof}
		This lemma can be found in [28, Proposition 3.2] and [7, Lemma 2.2].
	\end{proof}
	
	\section{The supplementary results and sketch of the proof}
	Let \(\varepsilon>0\), \(A>0\), \(1/2\le\eta\le1-\varepsilon\), and define the set
	\(\mathcal{Q}(a,\eta):=\{q: x^\eta < q \le 2x^\eta,\ \gcd(q,a)=1\}\). We consider the following question:
	What function \(C(\eta)\) satisfies the inequality\begin{align*}
		\Psi(x,x^{1/2}; q, a)=\sum_{\substack{n\leq x,P^+(n)\leq x^{1/2}\\n\equiv a\mkern-15mu\pmod{q}}}1\geq (C(\eta)-\varepsilon)\frac{x}{q}\tag{3.1}
	\end{align*}
	for almost all \(q\in\mathcal{Q}(a,\eta)\), while the number of exceptions is less than \(x^\eta (\log x)^{-A}\)?
	\begin{remark}
		Why is  \(\cdot/q\) used in formula (3.1) instead of \(\cdot/\varphi(q)\)? Note that\begin{align*}
			\frac{1}{\varphi(q)}\sum_{\substack{n<z\\(n,q)=1}}1=\frac{z}{q}+O\left(\frac{\tau(q)}{\varphi(q)}\right).
		\end{align*}
	\end{remark}
	
	\begin{theorem}
		For all sufficiently small $\varepsilon>0$, there exists $\delta>0$ such that the following holds. For $1/2<\eta<66/107-\varepsilon$ and $a\ll x^{\delta}$, statement (3.1) holds for some effective \(C(\eta)>0\). In particular, for $0.5<\eta<0.55$, we have\begin{align*}
			C(\eta)\geq & \int\int_{(u_1,u_2)\in\mathscr{E}_1}\frac{1}{u_1u_2}\mathrm{d}u_1\mathrm{d}u_2+\int\int\int_{(u_0,u_1,u_2)\in \mathscr{E}_2}\frac{1}{u_0u_1u_2}\mathrm{d}{u_0}\mathrm{d}{u_1}\mathrm{d}{u_2},
		\end{align*}
		where $\theta=7/32$,
		\begin{align*}
			\mathscr{E}_1=&\{(u_0,u_1):1-\eta>u_1>u_2,u_1+2u_2>1,8u_1-u_2>5-4\eta,\\&(7-2\theta)u_1+(3-4\theta)u_2>5-2\theta-2\eta\}
		\end{align*}
		and
		\begin{align*}
			\mathscr{E}_2=&\{(u_0,u_1,u_2):1-\eta>u_0+u_1>u_2>u_0>u_1>0,8(u_0+u_1)-u_2>5-4\eta,\\&(7-2\theta)(u_0+u_1)+(3-4\theta)u_2>5-2\theta-2\eta, 2u_0+u_2<1,u_0+2u_1+u_2>1\}.
		\end{align*}
	\end{theorem}
	\begin{remark}
		For any $y\in(x^{\varepsilon},x]$ and $P^+(n)<y$, the similar result holds. This gives many effective applications. For example, this can improve the results in [16] and also obtain effective constants.
	\end{remark}

	Before starting the proof, we need two lemmas. The first lemma is [19, Theorem 4.2].
	\setcounter{lemma}{3}
	\begin{lemma}
		For all sufficiently small $\varepsilon>0$, there exists $\delta>0$ such that the following holds. Let $a_1$, $a_2$ be coprime nonzero integers, and let $M$, $N$, $L$, $R$, $x>2$ satisfy\begin{align*}
			a_1,a_2\ll x^{\delta},\ MNL\asymp x,\ x^{(1-\varepsilon)/2}\ll R\ll x^{-5\varepsilon}NL\ll x^{2/3-11\varepsilon},\ N^9L^8\ll x^3R^4,
		\end{align*}
		\begin{align*}
			N\ll\frac{x^{1-2\varepsilon}}{R},\ N^4L^7\max(1,N/L)^{2\theta}\ll x^{2-17\varepsilon}R^2,\ N^{12-6\theta}L^{11-6\theta}\ll x^{6-4\theta-5\varepsilon}R^2,
		\end{align*}
		for $\theta=7/32$. Then for any $1$-bounded complex sequences $(\alpha_m)$, $(\beta_n)$, $(\gamma_{\ell})$, one has\begin{align*}
			\sum_{\substack{r\sim R\\(r,a_1a_2)=1}}\Bigg|\sum_{m\sim M}\sum_{n\sim N}\sum_{\ell\sim L}\alpha_m\beta_n\gamma_{\ell}\Bigg(1_{mn\ell\equiv a_1\overline{a_2}\mkern-15mu\pmod{r}}-\frac{1_{(mn\ell,r)=1}}{\varphi(r)}\omega_{\varepsilon}(mn\ell\overline{a_1}a_2;r)\Bigg)\Bigg|\ll\frac{x}{(\log x)^A},
		\end{align*}
		where $m\sim M$ denotes $M<m\leq 2M$, the constant implied in $\ll$ is depended on $\varepsilon$ and $A$, \begin{align*}
			\omega_\varepsilon(k;r):=\sum_{\substack{\chi\ \mathrm{primitive}\\\mathrm{cond}(\chi)|r\\1<\mathrm{cond}(\chi)\leq x^{\varepsilon}}}\chi(k).
		\end{align*}
	\end{lemma}
	
	The following lemma is the Pólya-Vinogradov inequality with an explicit dependence on the conductor.
	\begin{lemma}
		Let $\chi\mkern-7mu\pmod{q}$ be an Dirichlet character of conductor $r\neq 1$ with $r|q$. Let $M$, $N\geq 1$. We have\begin{align*}
			\sum_{M<n\leq M+N}\chi(n)\ll\tau(q/r)\sqrt{r}\log r.
		\end{align*}
	\end{lemma}
	\setcounter{subsection}{0}
	\subsection{Proof of Theorem 3.2}
	For $\eta\in(1/2,66/107-\varepsilon)$, let $R=x^{\eta}$. We write $M=x^{u_1}$, $N=x^{u_2}$, $L=x^{u_3}$. Let $\theta=7/32$. The conditions in Lemma 3.4 become\begin{align*}
		u_1+u_2+u_3=1, u_2+u_3>\eta,\ 9u_2+8u_3< 3+4\eta,\ u_2+\eta< 1,\ 4u_2+7u_3<2+2\eta,
	\end{align*}
	\begin{align*}
		(4+2\theta)u_2+(7-2\theta)u_3< 2+2\eta,\ (12-6\theta)u_2+(11-6\theta)u_3<6-4\theta+2\eta,
	\end{align*}
	where we omit some $\varepsilon$.
	Supposed $u_1>u_2>u_3>0$, the all conditions become
	\begin{align*}
		1-\eta>u_1>u_2,\ u_1+2u_2>1, \ 8u_1-u_2>5-4\eta,\ 
	\end{align*}
	\begin{align*}
		(7-2\theta)u_1+(3-4\theta)u_2>5-2\theta-2\eta.
	\end{align*}
	Now suppose that all the above conditions are satisfied.
	Let $m=p_1$, $n=p_2$ be two primes. For the $\delta$ in Lemma 3.4 and $a\ll x^{\delta}$, we have\begin{align*}
		&\sum_{\substack{r\sim R\\(r,a)=1}}\Bigg|\sum_{p_1\sim M}\sum_{p_2\sim N}\sum_{\ell\sim L}\Bigg(1_{p_1p_2\ell\equiv a\mkern-15mu\pmod{r}}-\frac{1_{(p_1p_2\ell,r)=1}}{\varphi(r)}\Bigg)\Bigg|\\&\ll\sum_{\substack{r\sim R\\(r,a)=1}}\Bigg|\sum_{p_1\sim M}\sum_{p_2\sim N}\sum_{\ell\sim L}\Bigg(1_{p_1p_2\ell\equiv a\mkern-15mu\pmod{r}}-\frac{1_{(p_1p_2\ell,r)=1}}{\varphi(r)}\omega_\varepsilon(p_1p_2\ell\overline{a};r)\Bigg)\Bigg|\\&\quad+\sum_{\substack{r\sim R\\(r,a)=1}}\sum_{\substack{\chi\ \mathrm{primitive}\\\mathrm{cond}(\chi)|r\\1<\mathrm{cond}(\chi)\leq x^{\varepsilon}}}\frac{1}{\varphi(r)}\Bigg|\sum_{p_1\sim M}\sum_{p_2\sim N}\sum_{\substack{\ell\sim L\\(p_1p_2\ell,r)=1}}\chi(p_1p_2\ell)\Bigg|
	\end{align*}
	We apply Lemma 3.4 to bound the first summation. For the second summation,
	we apply M\"obius inversion to remove the condition $(l,r)=1$ and get\begin{align*}
		&\sum_{\substack{r\sim R\\(r,a)=1}}\sum_{\substack{\chi\ \mathrm{primitive}\\\mathrm{cond}(\chi)|r\\1<\mathrm{cond}(\chi)\leq x^{\varepsilon}}}\frac{1}{\varphi(r)}\Bigg|\sum_{p_1\sim M}\sum_{p_2\sim N}\sum_{\substack{\ell\sim L\\(p_1p_2\ell,r)=1}}\chi(p_1p_2\ell)\Bigg|\\&\ll\sum_{\substack{r\sim R\\(r,a)=1}}\sum_{\sigma|r}\sum_{\substack{\chi\ \mathrm{primitive}\\\mathrm{cond}(\chi)|r\\1<\mathrm{cond}(\chi)\leq x^{\varepsilon}}}\frac{1}{\varphi(r)}\sum_{p_1\sim M}\sum_{p_2\sim N}\Bigg|\sum_{\substack{\ell\sim L/\sigma}}\chi(\ell)\Bigg|
	\end{align*}
	The contribution of $\sigma>x^{3\varepsilon}$ is negligible by trivial bound. Now we consider $\sigma<x^{3\varepsilon}$. Applying Lemma 3.5, we get\begin{align*}
		\sum_{\substack{\ell\sim L/\sigma}}\chi(\ell)\ll x^{\varepsilon}.
	\end{align*}
	Thus we can get\begin{align*}
		\sum_{\substack{r\sim R\\(r,a)=1}}\sum_{\substack{\chi\ \mathrm{primitive}\\\mathrm{cond}(\chi)|r\\1<\mathrm{cond}(\chi)\leq x^{\varepsilon}}}\frac{1}{\varphi(r)}\Bigg|\sum_{p_1\sim M}\sum_{p_2\sim N}\sum_{\substack{\ell\sim L\\(p_1p_2\ell,r)=1}}\chi(p_1p_2\ell)\Bigg|\ll\frac{x}{(\log x)^A},
	\end{align*}
	then\begin{align*}
		\sum_{\substack{r\sim R\\(r,a)=1}}\Bigg|\sum_{p_1\sim M}\sum_{p_2\sim N}\sum_{\ell\sim L}\Bigg(1_{p_1p_2\ell\equiv a\mkern-15mu\pmod{r}}-\frac{1_{(p_1p_2\ell,r)=1}}{\varphi(r)}\Bigg)\Bigg|\ll\frac{x}{(\log x)^A}.
	\end{align*}
	We only need to consider the  number of the type $p_1p_2\ell$. Obviously, the number of the type $p_1p_2\ell$ is\begin{align*}
		(1-O(\varepsilon))x\int\int_{(u_1,u_2)\in\mathscr{E}_1}\frac{1}{u_1u_2}\mathrm{d}u_1\mathrm{d}u_2,
	\end{align*}
	where\begin{align*}
		\mathscr{E}_1=&\{(u_1,u_2):       1-\eta >u_1>u_2, u_1+2u_2>1,8u_1-u_2>5-4\eta, \\&(7-2\theta)u_1+(3-4\theta)u_2>5-2\theta-2\eta\}.
	\end{align*}
	Thus we obtain an effective $C(\eta)>0$ for $1/2<\eta<66/107-\varepsilon$. Taking $m=p_1p_1'$ to be an almost-prime with $p_1=x^{u_0}$ and $p_1'=x^{u_1}$ ($n=p_2=x^{u_2}$ and $\ell=x^{u_3}$), the argument is analogous. 
	Let $u_2>u_0>u_1>u_3$ and $u_2+2u_0<1$ (compared with $u_1+2u_2>1$ above) to avoid repetition.
	Omitting the details, for $0.5<\eta<0.55$, we can obtain the $C(\eta)$ in Theorem 3.2. We complete the proof.
	\subsection{Sketch of Theorem 1.4 and our improvements}
	In this paper, we mainly focus on the estimation of $\#\{n<x:P^+(n)<P^+(n+1)\leq x^{1-c}\}$. (We also obtain a better result for the case $P^+(n)\geq P^+(n+1)> x^{1-c}$.) We optimize the calculation (giving an improvement of approximately 0.04), and advance the idea of L\"u and Wang [17] in several respects.
	
	For the case $P^+(n)<x^{1/2}/(\log x)^B$, $p_1|n+1$ with $\max(x^{\delta_1},P^+(n))<p_1<x^{1/2}/(\log x)^B$, in [17, (3.3)] (also see [28, p.367]), L\"u and Wang used a sieve weight\begin{align*}
		\frac{\log z}{\log x}\omega(n;y,z):=\frac{\log z}{\log x}\sum_{\substack{z<p\leq y\\p|n}}1\leq 1.
	\end{align*}
	However, when \(n+1\) has only one prime divisor $p_1\in(x^{\delta_1},x^{1/2}/(\log x)^B)$, this weight leads to many losses. From our observations, in \S4.1 and \S4.3, we abandon this weight, and when $n+1$ has two large prime divisors, we apply the generalized Bombieri-Vinogradov theorem and Rosser-Iwaniec sieve to obtain an upper bound. (giving an improvement of approximately 0.015)
	
	For the case $P^+(n+1)>x^{1/2}$, L\"u and Wang showed their lower bounds for sums of the form \begin{align*}
		&\sum_{x^{1-\eta_1}<p<x^{1-\eta_{2}}}\sum_{dp\leq x}1-\sum_{x^{1-\eta_1}<p'<x}\sum_{x^{1-\eta_1}<p<x^{1-\eta_{2}}}\sum_{\substack{dp\leq x\\dp-1\equiv0\mkern-15mu\pmod{p'}}}1\\&\geq\Bigg(\frac{1-2\eta_1 }{2(1-\eta_1)}\log\frac{1-\eta_2}{1-\eta_1}+o(1)\Bigg)x+\frac{S_B}{\log x},
	\end{align*}
	where \begin{align*}
		S_B=&\frac{1}{1-\eta_1}\sum_{x^{1/2}/(\log x)^B<p'\leq x^{1-\eta_1}}\log p'\sum_{x^{1-\eta_1}<p\leq x^{1-\eta_2}}\sum_{\substack{dp\leq x\\dp-1\equiv 0\mkern-15mu\pmod{p'}}}1\\&+\sum_{x^{1-\eta_1}<p'<x}\left(\frac{1}{1-\eta_1}\log p'-\log x\right)\sum_{x^{1-\eta_1}<p\leq x^{1-\eta_2}}\sum_{\substack{dp\leq x\\dp-1\equiv 0\mkern-15mu\pmod{p'}}}1.
	\end{align*}
	For the sum $S_B$ (as an error term), they use a trivial result $S_B\geq 0$. Following their works, in \S4.2, an important improvement in our argument is that we introduce a new parameter $t_1\in(0,\eta_1)$ (depending on $\eta_1$) to improve the main term (giving an improvement of approximately 0.02, see (4.8) and (4.9), also see (4.19), (6.2)), and $S_B$ becomes\begin{align*}
		S_B&=\frac{1}{1-\eta_1+t_1}\sum_{x^{1/2}/(\log x)^B<p'<x^{1-\eta_1}}\log p'\sum_{x^{1-\eta_1}<p<x^{1-\eta_2}}\sum_{\substack{dp\leq x\\dp-1\equiv0\mkern-15mu\pmod{p'}}}1\\&\quad+\sum_{x^{1-\eta_1}<p'<x}\left(\frac{1}{1-\eta_1+t_1}\log p'-\log x\right)\sum_{x^{1-\eta_1}<p<x^{1-\eta_2}}\sum_{\substack{dp\leq x\\dp-1\equiv0\mkern-15mu\pmod{p'}}}1.
	\end{align*}
	Our goal is to find the largest $t_1$ such that $S_B\geq0$. We observe that for $0<t_2<\eta_1-t_1$, one has
	\begin{align*}
		S_B&\geq\frac{1}{2}\left(\frac{1}{1-\eta_1+t_1}-\frac{1}{1-\eta_1+t_1+t_2}\right)\log \frac{1-\eta_2}{1-\eta_1}x\log x+o(x\log x)\\&\quad-\sum_{x^{1-\eta_1}<p'<x^{1-\eta_1+t_1+t_2}}\left(\log x-\frac{1}{1-\eta_1+t_1+t_2}\log p'\right)\sum_{x^{1-\eta_1}<p<x^{1-\eta_2}}\sum_{\substack{dp\leq x\\dp-1\equiv0\mkern-15mu\pmod{p'}}}1.
	\end{align*}
	Then we apply Rosser-Iwaniec sieve to obtain an upper bound of the final summation (see (4.10)-(4.14)), and give the conditions on $t_1$. (The parameter $t_1$ also plays an important role in many other problems. For example, we can see the $\delta=t_1/(c+t_1)$ in Theorem 1.6. For $c=1/2$, the $\delta$ is larger than $0.12$ (better value by Remark 6.2), and the expected value of $\delta$ is \(\approx0.3\).)
	
	However, for $\eta_1\approx 1/2$, the lower bound of the main term above is not a good result since $1-2\eta_1\approx0$. Note that for $\eta_1\in(0,1/2)$, one has\begin{align*}
		&\sum_{x^{1-\eta_1}<p<x^{1-\eta_{2}}}\sum_{dp\leq x}1-\sum_{x^{1-\eta_1}<p'<x}\sum_{x^{1-\eta_1}<p<x^{1-\eta_{2}}}\sum_{\substack{dp\leq x\\dp-1\equiv0\mkern-15mu\pmod{p'}}}1\\&\geq \sum_{x^{1-\eta_1}<p<x^{1-\eta_{2}}}\sum_{\substack{n\leq x-1, P^+(n)\leq x^{1/2}\\p|n+1}}1.
	\end{align*}
	Thus we directly consider the lower bound of $\Psi(x,x^{1/2};p,-1)=\#\{n<x: P^+(n)<x^{1/2},p|n+1\}$ and apply Theorem 3.2. (giving an improvement of approximately $0.003$)

	\section{Proof of Theorem 1.4}
	
	We start from the following expression
	\begin{align*}
		\sum_{\substack{n<x\\P^+(n)<P^+(n+1)}}1&=\sum_{\substack{n<x\\P^+(n+1)>x^{1-c}}}1-\sum_{\substack{n<x\\P^+(n)\geq P^+(n+1)>x^{1-c}}}1+\sum_{\substack{n<x\\P^+(n)<P^+(n+1)\leq x^{1-c}}}1\\&=\mathscr{S}_A-\mathscr{S}_B+\mathscr{S}_C,\tag{4.1}
	\end{align*}
	where $0<c\leq 1/2$ is a parameter. To estimate $\mathscr{S}_A$ and $\mathscr{S}_B$, we have the following results.
	\begin{lemma}
		For $0<c<2/7-\varepsilon$, we have
		\begin{align*}
			&\mathscr{S}_A=x\log\frac{1}{1-c}+o(x),\\&\mathscr{S}_B\leq x\left(2\int_0^c\log\left(\frac{1-t}{1-c}\right)\frac{\mathrm{d}t}{4/7-t}+o(1)\right).
		\end{align*}
	\end{lemma}
	\setcounter{theorem}{1}
	\begin{remark}
		Let $c=0.1348$. The upper bound of $\mathscr{S}_B$ becomes $\approx0.0382x$. Wang [28] obtained\begin{align*}
			\mathscr{S}_B\leq \left(2\int_0^c\log\left(\frac{1}{1-t}\right)\frac{\mathrm{d}t}{4/7-t}+o(1)\right)x\approx 0.0398x.
		\end{align*}
		After correcting Wang’s proof, we should choose the sieve weight $\lambda^+$ to have a larger level when applying sieve methods in most cases. (This is not the key point of this paper.)
	\end{remark}
	\begin{proof}
		The estimation of $\mathscr{S}_A$ is trivial. For $\mathscr{S}_B$, we write $dp=n+1$, and we have\begin{align*}
			\mathscr{S}_B=& \sum_{\substack{n<x\\P^+(n)\geq P^+(n+1)>x^{1-c}}}1\leq \sum_{i=1}^N\sum_{x^{1-\eta_i}<p'<x}\sum_{x^{1-\eta_i}<p<x^{1-\eta_{i+1}}}\sum_{\substack{dp<x\\dp-1\equiv 0\mkern-15mu\pmod{p'}}}1+o(x),
		\end{align*}
		where $2/7-\varepsilon>  c=\eta_1>\eta_2>\cdots>\eta_N>\eta_{N+1}=0$. For each $i$, applying Rosser-Iwaniec sieve (see (4.10)-(4.14)), we claim that\begin{align*}
			\sum_{x^{1-\eta_i}<p'<x}\sum_{x^{1-\eta_i}<p<x^{1-\eta_{i+1}}}\sum_{\substack{dp<x\\dp-1\equiv 0\mkern-15mu\pmod{p'}}}1\leq \left(2\int_{0}^{\eta_i}\frac{1}{4/7-t}\mathrm{d}t\log\frac{1-\eta_{i+1}}{1-\eta_i}+o(1)\right)x.
		\end{align*}
		Thus we have\begin{align*}
			\mathscr{S}_B\leq \left(2\sum_{i=1}^N\int_{0}^{\eta_i}\frac{1}{4/7-t}\mathrm{d}t\log\frac{1-\eta_{i+1}}{1-\eta_i}+o(1)\right)x.
		\end{align*}
		For any $\nu>0$, there exists $\Delta=\Delta(\nu)>0$ such that if $\max_i|\eta_{i+1}-\eta_i|<\Delta$, one has\begin{align*}
			\mathscr{S}_B&\leq\left(2\int_0^c\int_{0}^{\eta}\frac{1}{(4/7-t)(1-\eta)}\mathrm{d}t\mathrm{d}\eta+\nu+o(1)\right)x\\&=\left(2\int_0^c\log\left(\frac{1-t}{1-c}\right)\frac{\mathrm{d}t}{4/7-t}+\nu+o(1)\right)x.
		\end{align*}
		We complete the proof.
	\end{proof}
	
	Now we begin to estimate $\mathscr{S}_C$. In this paper, we present a new lower bound for $\mathscr{S}_C$. Define
	\begin{align*}
		P(y,z):=\prod_{y<p\leq z}p
	\end{align*}
	and \begin{align*}
		S(x,y):=\{n\leq x:P^+(n)\leq y\},\quad S^+(x;y,z):=\{n\leq x:y<P^+(n)\leq z\}.
	\end{align*}
	It is not difficult to prove that for $c\in(0,2/7-\varepsilon)$ and $\delta_1\in(1/3,1/2)$ (for convenience, $c=0.1348$, $\delta_1=0.417$), one has
	\begin{align*}
		\mathscr{S}_C=&\sum_{\substack{n<x\\P^+(n)<P^+(n+1)\leq x^{1-c}}}1\\\geq&\sum_{\substack{n\in S(x,x^{\delta_1})\\(n+1,P(x^{\delta_1},x^{1/2}/(\log x)^B))>1}}1+\sum_{\substack{n\in S^+(x;x^{\delta_1},x^{1/2}/(\log x)^B)\\(n+1,P(P^+(n),x^{1/2}/(\log x)^B))>1}}1\\&+\sum_{\substack{n\in S(x,x^{1-\delta_1})\\\max(x^{1/2},P^+(n))<P^+(n+1)<x^{1-c}}}1+\sum_{\substack{n\in S^+(x;x^{1-\delta_1},x^{1-c})\\P^+(n)<P^+(n+1)<x^{1-c}}}1\\&-\sum_{\substack{n\in S(x,x^{1/2}/(\log x)^B)\\(n+1,P(x^{\delta_1},x^{1/2}/(\log x)^B))>1\\x^{1/2}<P^+(n+1)<x^{1-\delta_1}}}1\\=&\mathscr{S}'_1+\mathscr{S}'_2+\mathscr{S}'_3+\mathscr{S}'_4-\mathscr{S}'_5,\tag{4.2}
	\end{align*}
	where $\mathscr{S}_5$ contains mainly the terms that appear repeatedly in $\mathscr{S}'_1$, $\mathscr{S}'_2$ and $\mathscr{S}'_3$.
	\subsection{Estimation of $\mathscr{S}'_1$ and $\mathscr{S}'_2$}Firstly, for $\mathscr{S}'_1$, we classify integers $n+1$ based on whether \(n+1\) has one or two required prime divisors,
	and we obtain \begin{align*}
		\mathscr{S}'_1=&\sum_{\substack{n\in S(x,x^{\delta_1})\\(n+1,P(x^{\delta_1},x^{1/2}/(\log x)^B))>1}}1\\=&\sum_{x^{\delta_1}<p\leq x^{1/2}/(\log x)^B}\sum_{\substack{n\in S(x,x^{\delta_1})\\n\equiv-1\mkern-15mu\pmod{p}}}1-\frac{1}{2}\sum_{x^{\delta_1}<p_1,p_2\leq x^{1/2}/(\log x)^B}\sum_{\substack{n\in S(x,x^{\delta_1})\\p_1p_2|n+1}}1.
	\end{align*}
	We apply Lemma 2.1 and Lemma 2.2 to estimate the frist summation and obtain\begin{align*}
		\mathscr{S}'_1=&\sum_{x^{\delta_1}<p\leq x^{1/2}/(\log x)^B}\frac{1}{p-1}\Psi(x,x^{\delta_1})+o(x)-\frac{1}{2}\sum_{x^{\delta_1}<p_1,p_2\leq x^{1/2}/(\log x)^B}\sum_{\substack{n\in S(x,x^{\delta_1})\\p_1p_2|n+1}}1\\=&\Bigg(\rho\left(\frac{1}{\delta_1}\right)\log \frac{1}{2\delta_1}+o(1)\Bigg)x-\frac{1}{2}\sum_{x^{\delta_1}<p_1,p_2\leq x^{1/2}/(\log x)^B}\sum_{\substack{n\in S(x,x^{\delta_1})\\p_1p_2|n+1}}1.
	\end{align*}
	For $\mathscr{S}'_2$, we can get\begin{align*}
		\mathscr{S}'_2=&\sum_{\substack{n\in S^+(x;x^{\delta_1},x^{1/2}/(\log x)^B)\\(n+1,P(P^+(n),x^{1/2}/(\log x)^B))>1}}1\\\geq& \sum_{\substack{n\in S^+(x;x^{\delta_1},x^{1/2}/(\log x)^B)}}\sum_{\substack{p|n+1\\P^+(n)<p<x^{1/2}/(\log x)^B)}}1-\mathcal{R}_0\\&\\\geq& \sum_{i=1}^N\sum_{\substack{n\in S^+(x;x^{\delta_i},x^{\delta_{i+1}})}}\sum_{\substack{p|n+1\\x^{\delta_{i+1}}<p<x^{1/2}/(\log x)^B)}}1-\mathcal{R}_0,
	\end{align*}
	where $\delta_1<\delta_2<\cdots<\delta_{N+1}<1/2$,
	\begin{align*}
		\mathcal{R}_0=\frac{1}{2}\sum_{\substack{n\in S^+(x;x^{\delta_1},x^{1/2}/(\log x)^B)}}\sum_{\substack{P^+(n)<p_1,p_2\leq x^{1/2}/(\log x)^B\\p_1p_2|(n+1)}}1.
	\end{align*}
	Since\begin{align*}
		\sum_{\substack{n\in S^+(x;x^{\delta_i},x^{\delta_{i+1}})}}\sum_{\substack{p|n+1\\x^{\delta_{i+1}}<p<x^{1/2}/(\log x)^B)}}1=\sum_{\substack{x^{\delta_{i+1}}<p<x^{1/2}/(\log x)^B)}}\Bigg(\sum_{\substack{n\in S(x,x^{\delta_{i+1}})\\n\equiv -1\mkern-15mu\pmod{p}}}1-\sum_{\substack{n\in S(x,x^{\delta_{i}})\\n\equiv -1\mkern-15mu\pmod{p}}}1\Bigg),
	\end{align*}
	a similar argument to that of $\mathscr{S}_1$ allows us to deduce that\begin{align*}
		\mathscr{S}'_2\geq& \sum_{i=1}^{N}\left(\rho\left(\frac{1}{\delta_{i+1}}\right)-\rho\left(\frac{1}{\delta_{i}}\right)\right)\left(\log \frac{1}{2\delta_{i+1}}\right)x+o(x)\\&-\frac{1}{2}\sum_{\substack{x^{\delta_1}<p_1,p_2\leq x^{1/2}/(\log x)^B}}\sum_{\substack{n\in S^+(x;x^{\delta_1},x^{1/2}/(\log x)^B)\\p_1p_2|n+1}}1.
	\end{align*}
	Then there exists $\Delta>0$ such that if $\max(\max_{i}|\delta_{i+1}-\delta_i|,1/2-\delta_{N+1})<\Delta$, one has
	\begin{align*}
		\mathscr{S}'_{1}+\mathscr{S}'_{2}\geq& x\left(\rho\left(\frac{1}{\delta_1}\right)\log\frac{1}{2\delta_1}+\sum_{i=1}^{N}\left(\rho\left(\frac{1}{\delta_{i+1}}\right)-\rho\left(\frac{1}{\delta_{i}}\right)\right)\log\frac{1}{2\delta_{i}}+o(1)\right)\\&-\frac{1}{2}\sum_{x^{\delta_1}<p_1,p_2<x^{1/2}/(\log x)^B}\sum_{\substack{n\in S(x,x^{1/2}/(\log x)^B)\\p_1p_2|n+1}}1\\\geq &\left(\int_{\delta_1}^{1/2}\rho\left(\frac{1}{t}\right)\frac{1}{t}\mathrm{d}t-\nu+o(1)\right)x-\mathcal{R},\tag{4.3}
	\end{align*}
	where $\nu=0.000001$,\begin{align*}
		\mathcal{R}=\frac{1}{2}\sum_{x^{\delta_1}<p_1,p_2<x^{1/2}/(\log x)^B}\sum_{\substack{n\in S(x,x^{1/2}/(\log x)^B)\\p_1p_2|n+1}}1.
	\end{align*}
	We leave the estimation of $\mathcal{R}$ for \S4.3.
	\subsection{Estimation of $\mathscr{S}'_3$ and $\mathscr{S}'_4$}
	Noting that $P^+(n+1)>x^{1/2}$, so we can not use the theorem of Bombieri-Vinogradov type for friable integers as we have done for $\mathscr{S}'_1$ and $\mathscr{S}'_2$.
	
	Firstly, we partition the cases according to the largest prime divisors of $n$ and \(n+1\). We get
	\begin{align*}
		\mathscr{S}'_3+\mathscr{S}'_4=&\sum_{\substack{n\in S(x,x^{1-\delta_1})\\\max(x^{1/2},P^+(n))<P^+(n+1)<x^{1-c}}}1+\sum_{\substack{n\in S^+(x;x^{1-\delta_1},x^{1-c})\\P^+(n)<P^+(n+1)<x^{1-c}}}1\\=&\sum_{\substack{n\in S(x,x^{1-c})\\\max(x^{1/2},P^+(n))<P^+(n+1)<x^{1-c}}}1\\\geq &\sum_{\substack{n\in S(x,x^{1-\eta_1})\\x^{1-\eta_1}<P^+(n+1)<x^{1-c}}}1+\sum_{i=1}^{N-1}\sum_{\substack{n\in S^+(x;x^{1-\eta_i},x^{1-\eta_{i+1}})\\x^{1-\eta_{i+1}}<P^+(n+1)<x^{1-c}}}1\\=&\mathscr{S}_{C_1}'+\sum_{i=1}^{N-1}\mathscr{S}_{C_{i+1}}',
	\end{align*}
	where $1/2>\eta_1>\eta_2>\cdots>\eta_N>\eta_{N+1}=c$. Following L\"u and Wang's works,
	we have (see [17, p.412-413])
	\begin{align*}
		\mathscr{S}_{C_1}'+\mathscr{S}_{C_2}'=\widetilde{\mathscr{S}_{C_1}}+\widetilde{\mathscr{S}_{C_2}'}+o(x),
	\end{align*}
	where
	\begin{align*}
		\widetilde{\mathscr{S}_{C_1}}&=\sum_{x^{1-\eta_1}<p< x^{1-\eta_2}}\sum_{dp\leq x}1-\sum_{x^{1-\eta_1}<p'<x}\sum_{x^{1-\eta_1}<p< x^{1-\eta_2}}\sum_{\substack{dp\leq x\\dp-1\equiv0\mkern-15mu\pmod{p'}}}1,\\
		\widetilde{\mathscr{S}_{C_2}'}&=\sum_{x^{1-\eta_2}<p< x^{1-c}}\sum_{dp\leq x}1-\sum_{x^{1-\eta_2}<p'<x}\sum_{x^{1-\eta_2}<p< x^{1-c}}\sum_{\substack{dp\leq x\\dp-1\equiv0\mkern-15mu\pmod{p'}}}1\\&=\sum_{\substack{n\in S(x,x^{1-\eta_2})\\x^{1-\eta_2}<P^+(n+1)<x^{1-c}}}1+O(1).
	\end{align*}
	Repeating the above procedure, we have
	\begin{align*}
		\mathscr{S}'_3+\mathscr{S}'_4&\geq\mathscr{S}_{C_1}'+\sum_{i=1}^{N-1}\mathscr{S}_{C_{i+1}}'=\widetilde{\mathscr{S}_{C_1}}+\Big(\widetilde{\mathscr{S}_{C_2}'}+\mathscr{S}'_{C_{3}}\Big)+\sum_{i=3}^{N-1}\mathscr{S}_{C_{i+1}}'+o(x)\\&=\sum_{i=1}^2\widetilde{\mathscr{S}_{C_i}}+\Big(\widetilde{\mathscr{S}_{C_3}'}+\mathscr{S}_{C_{4}}'\Big)+\sum_{i=4}^{N-1}\mathscr{S}_{C_{i+1}}'+o(x)\\&\cdots\\&=\sum_{i=1}^{N}\widetilde{\mathscr{S}_{C_i}}+o(x),\tag{4.5}
	\end{align*}
	where $\widetilde{\mathscr{S}_{C_N}}=\widetilde{\mathscr{S}_{C_N}'}$,
	\begin{align*}
		&\widetilde{\mathscr{S}_{C_i}}=\sum_{x^{1-\eta_i}<p<x^{1-\eta_{i+1}}}\sum_{dp\leq x}1-\sum_{x^{1-\eta_i}<p'<x}\sum_{x^{1-\eta_i}<p<x^{1-\eta_{i+1}}}\sum_{\substack{dp\leq x\\dp-1\equiv0\mkern-15mu\pmod{p'}}}1,\\&\widetilde{\mathscr{S}_{C_{i}}'}=\sum_{x^{1-\eta_{i}}<p<x^{1-c}}\sum_{dp\leq x}1-\sum_{x^{1-\eta_{i}}<p'<x}\sum_{x^{1-\eta_{i}}<p<x^{1-c}}\sum_{\substack{dp\leq x\\dp-1\equiv0\mkern-15mu\pmod{p'}}}1,
	\end{align*}
	for $1\leq i\leq N$.	
	Then we begin to estimate $\widetilde{\mathscr{S}_{C_1}}$. (The estimation of $\widetilde{\mathscr{S}_{C_i}}$ is similar to that of $\widetilde{\mathscr{S}_{C_1}}$.)
	Using the following fact that (see [17, (3.12) and (3.15)-(3.18)])
	\begin{align*}
		\sum_{x^{1-\alpha}<p<x^{1-\beta}}\sum_{dp\leq x}1=\frac{1}{\log x}\sum_{p'<x}\log p'\sum_{x^{1-\alpha}<p<x^{1-\beta}}\sum_{\substack{dp\leq x\\dp-1\equiv0\mkern-15mu\pmod{p'}}}1+o(x)
	\end{align*}
	and \begin{align*}
		\sum_{p'<x^{1/2}/(\log x)^B}\log p'\sum_{x^{1-\alpha}<p<x^{1-\beta}}\sum_{\substack{dp\leq x\\dp-1\equiv0\mkern-15mu\pmod{p'}}}1=\left(\frac{1}{2}\log \frac{1-\beta}{1-\alpha}+o(1)\right)x\log x,\tag{4.6}
	\end{align*}
	we get
	\begin{align*}
		\sum_{x^{1/2}/(\log x)^B<p'<x}\log p'\sum_{x^{1-\alpha}<p<x^{1-\beta}}\sum_{\substack{dp\leq x\\dp-1\equiv0\mkern-15mu\pmod{p'}}}1=\left(\frac{1}{2}\log \frac{1-\beta}{1-\alpha}+o(1)\right)x\log x,\tag{4.7}
	\end{align*}
	for $0<\beta<\alpha<1/2$. Thus for $t_1\in(0,\eta_1)$, we have
	\begin{align*}
		\widetilde{\mathscr{S}_{C_1}}=&\sum_{x^{1-\eta_1}<p<x^{1-\eta_{2}}}\sum_{dp\leq x}1-\sum_{x^{1-\eta_1}<p'<x}\sum_{x^{1-\eta_1}<p<x^{1-\eta_{2}}}\sum_{\substack{dp\leq x\\dp-1\equiv0\mkern-15mu\pmod{p'}}}1\\=&\frac{1}{1-\eta_1+t_1}\frac{1}{\log x}\sum_{p'<x}\log p'\sum_{x^{1-\eta_1}<p<x^{1-\eta_2}}\sum_{\substack{dp\leq x\\dp-1\equiv0\mkern-15mu\pmod{p'}}}1\\&-\sum_{x^{1-\eta_1}<p'<x}\sum_{x^{1-\eta_1}<p\leq x^{1-\eta_2}}\sum_{\substack{dp\leq x\\dp-1\equiv0\mkern-15mu\pmod{p'}}}1\\&+\left(1-\frac{1}{1-\eta_1+t_1}\right)\sum_{x^{1-\eta_1}<p<x^{1-\eta_2}}\sum_{dp\leq x}1+o(x)\\=&\frac{1}{1-\eta_1+t_1}\frac{1}{\log x}\sum_{p'<x^{1-\eta_1}}\log p'\sum_{x^{1-\eta_1}<p<x^{1-\eta_2}}\sum_{\substack{dp\leq x\\dp-1\equiv0\mkern-15mu\pmod{p'}}}1\\&+\frac{1}{\log x}\sum_{x^{1-\eta_1}<p'<x}\left(\frac{1}{1-\eta_1+t_1}\log p'-\log x\right)\sum_{x^{1-\eta_1}<p<x^{1-\eta_2}}\sum_{\substack{dp\leq x\\dp-1\equiv0\mkern-15mu\pmod{p'}}}1
		\\&+
		\left(1-\frac{1}{1-\eta_1+t_1}\right)\sum_{x^{1-\eta_1}<p<x^{1-\eta_2}}\sum_{dp\leq x}1+o(x)\\=&\frac{1}{\log x}(S_A+S_B)+S_C+o(x),\tag{4.8}
	\end{align*}
	where
	\begin{align*}
		&S_A=\frac{1}{1-\eta_1+t_1}\sum_{p'<x^{1/2}/(\log x)^B}\log p'\sum_{x^{1-\eta_1}<p<x^{1-\eta_2}}\sum_{\substack{dp\leq x\\dp-1\equiv0\mkern-15mu\pmod{p'}}}1\\&\quad=\Bigg(\frac{1}{2(1-\eta_1+t_1)}\log\frac{1-\eta_2}{1-\eta_1}+o(1)\Bigg)x\log x,
		\\&S_B=\frac{1}{1-\eta_1+t_1}\sum_{x^{1/2}/(\log x)^B<p'<x^{1-\eta_1}}\log p'\sum_{x^{1-\eta_1}<p<x^{1-\eta_2}}\sum_{\substack{dp\leq x\\dp-1\equiv0\mkern-15mu\pmod{p'}}}1\\& \qquad+\sum_{x^{1-\eta_1}<p'<x}\left(\frac{1}{1-\eta_1+t_1}\log p'-\log x\right)\sum_{x^{1-\eta_1}<p<x^{1-\eta_2}}\sum_{\substack{dp\leq x\\dp-1\equiv0\mkern-15mu\pmod{p'}}}1,\\&S_C=\left(1-\frac{1}{1-\eta_1+t_1}\right)\sum_{x^{1-\eta_1}<p<x^{1-\eta_2}}\sum_{dp\leq x}1\\&\quad=\Bigg(\left(1-\frac{1}{1-\eta_1+t_1}\right)\log\frac{1-\eta_2}{1-\eta_1}+o(1)\Bigg)x.
	\end{align*}
	Different from (3.13) and (3.14) in [17], we introduce a new parameter $t_1\in(0,\eta_1)$ (depending on $\eta_1$). We want to find the largest $t_1$ such that $S_B\geq0$. For $t_2\in(0,\eta_1-t_1)$, by (4.7), we have
	\begin{align*}
		S_B&=\frac{1}{1-\eta_1+t_1}\sum_{x^{1/2}/(\log x)^B<p'<x^{1-\eta_1}}\log p'\sum_{x^{1-\eta_1}<p<x^{1-\eta_2}}\sum_{\substack{dp\leq x\\dp-1\equiv0\mkern-15mu\pmod{p'}}}1\\&\quad+\sum_{x^{1-\eta_1}<p'<x}\left(\frac{1}{1-\eta_1+t_1}\log p'-\log x\right)\sum_{x^{1-\eta_1}<p<x^{1-\eta_2}}\sum_{\substack{dp\leq x\\dp-1\equiv0\mkern-15mu\pmod{p'}}}1\\&\geq\left(\frac{1}{1-\eta_1+t_1}-\frac{1}{1-\eta_1+t_1+t_2}\right)\sum_{x^{1/2}/(\log x)^B<p'<x}\log p'\sum_{x^{1-\eta_1}<p<x^{1-\eta_2}}\sum_{\substack{dp\leq x\\dp-1\equiv0\mkern-15mu\pmod{p'}}}1\\&\quad+\sum_{x^{1-\eta_1}<p'<x^{1-\eta_1+t_1+t_2}}\left(\frac{1}{1-\eta_1+t_1+t_2}\log p'-\log x\right)\sum_{x^{1-\eta_1}<p<x^{1-\eta_2}}\sum_{\substack{dp\leq x\\dp-1\equiv0\mkern-15mu\pmod{p'}}}1\\&=\frac{1}{2}\left(\frac{1}{1-\eta_1+t_1}-\frac{1}{1-\eta_1+t_1+t_2}\right)\log \frac{1-\eta_2}{1-\eta_1}x\log x+o(x\log x)\\&\quad-\sum_{x^{1-\eta_1}<p'<x^{1-\eta_1+t_1+t_2}}\left(\log x-\frac{1}{1-\eta_1+t_1+t_2}\log p'\right)\sum_{x^{1-\eta_1}<p<x^{1-\eta_2}}\sum_{\substack{dp\leq x\\dp-1\equiv0\mkern-15mu\pmod{p'}}}1.\tag{4.9}
	\end{align*}

	Subsequently, we apply the Rosser-Iwaniec sieve to derive an upper bound estimate for the final summation of the RHS of (4.9).
	Write\begin{align*}
		\mathcal{B}(m)=\{dp\leq x:x^{1-\eta_1}<p<x^{1-\eta_2},dp-1\equiv0\mkern-15mu\pmod{m}\}.
	\end{align*}
	Then we have\begin{align*}
		&\sum_{x^{1-\eta_1}<p'<x^{1-\eta_1+t_1+t_2}}\left(\log x-\frac{1}{1-\eta_1+t_1+t_2}\log p'\right)\sum_{x^{1-\eta_1}<p<x^{1-\eta_2}}\sum_{\substack{dp\leq x\\dp-1\equiv0\mkern-15mu\pmod{p'}}}1\\&\leq\sum_{x^{\eta_1-t_1-t_2}/(\log x)^B<m<x^{\eta_1}}\left(\log x-\frac{1}{1-\eta_1+t_1+t_2}\log (x/(m(\log x)^B))\right)\\&\quad\times|\{dp\in \mathcal{B}(m):(dp-1)/m\  \text{is a prime number}\}|+O\left(\frac{x}{(\log x)^A}\right).\tag{4.10}
	\end{align*}
	It is not difficult to deduce that $2|dm$.  By Lemma 2.4, we have\begin{align*}
		&|\{dp\in \mathcal{B}(m):(dp-1)/m\  \text{is a prime number}\}|\\&\leq |\{dp:x^{\eta_2}/(\log x)^B<d\leq x^{\eta_1} ,dp\leq x,2|dm,((dp-1)/m,P(z))=1\}|\\&\leq \sum_{d'}\lambda_{d'}^+|\{dp:x^{\eta_2}/(\log x)^B<d\leq x^{\eta_1},dp\leq x,2|dm,(dp-1)/m\equiv0\mkern-15mu\pmod{d'}\}|\\&  =\sum_{d'}\frac{\lambda_{d'}^+}{\varphi(md')}\sum_{\substack{x^{\eta_2}/(\log x)^B<d\leq x^{\eta_1}\\2|dm,(d,md')=1}}\pi(x/d)\\&\quad+\sum_{d'}\lambda_{d'}^+\sum_{\substack{x^{\eta_2}/(\log x)^B<d\leq x^{\eta_1}\\2|dm,(d,md')=1}}\Big(\pi(x;d,1,md')-\frac{\pi(x/d)}{\varphi(md')}\Big),\tag{4.11}
	\end{align*}
	where $d'|P(z)$, $P(z)=\prod_{p\leq z}p$, $d'<D=z^2$ and $\lambda_{d'}^+*1\geq \mu*1$. 
	Thanks to the arguments of Iwaniec [8, Theorem 1 and Theorem 4], we can choose an appropriate sieve weight $\lambda_{d'}^+$ and the sum $\sum_{d'}\lambda_{d'}^+r(\mathcal{A},d')$ can be
	rearranged and transformed to the following flexible form:
	\begin{align*}
		\sum_{h<\exp(8/\varepsilon^2)}\sum_{d'|P(z)}\lambda^+(d',h)r(\mathcal{A},d'),
	\end{align*}
	where the coefficients $\lambda^+(d',h)=O(1)$ vanish for $d'>D$. Especially, $\lambda^+(d',h)$
	are well factorable of level $D$. Put $m\in(x^{\eta_1-t_1-t_2}/(\log x)^B,x^{\eta_1})$ in dynadic range $m\sim\mathcal{M}$. We choose
	\begin{align*}
		D=&\left\{ \begin{aligned}
			&x^{4/7-\varepsilon}/\mathcal{M}, \quad&& \eta_1<2/7-\varepsilon;\\&x^{1/2}/(\mathcal{M}(\log x)^B),\quad&& \text{otherwise}.
		\end{aligned}\right.
	\end{align*} 
	The final remainder term is
	\begin{align*}
		&\log x\sum_{h<\exp(8/\varepsilon^2)}\sum_{m\sim\mathcal{M}}f_1(m)\sum_{d'}\lambda^+(d',h)\sum_{\substack{x^{\eta_2}/(\log x)^B<d\leq x^{\eta_1}\\2|dm,(d,md')=1}}\Big(\pi(x;d,1,md')-\frac{\pi(x/d)}{\varphi(md')}\Big)\\=&\log x\sum_{h<\exp(8/\varepsilon^2)}\sum_{m\sim\mathcal{M}}f_1(m)\sum_{d'}\lambda^+(d',h)\sum_{\substack{x^{\eta_2}/(\log x)^B<d\leq x^{\eta_1}\\2|d,(d,md')=1}}\Big(\pi(x;d,1,md')-\frac{\pi(x/d)}{\varphi(md')}\Big)\\&+\log x\sum_{h<\exp(8/\varepsilon^2)}\sum_{\substack{m\sim\mathcal{M}\\2|m}}f_1(m)\sum_{d'}\lambda^+(d',h)\sum_{\substack{x^{\eta_2}/(\log x)^B<d\leq x^{\eta_1}\\(d,md')=1}}\Big(\pi(x;d,1,md')-\frac{\pi(x/d)}{\varphi(md')}\Big),
	\end{align*}
	where \begin{align*}
		f_1(m)=1-\frac{1}{1-\eta_1+t_1+t_2}\frac{\log (x/(m(\log x)^B))}{\log x}.
	\end{align*} 
	Note that for $2\mathcal{M}<x^{2/7-\varepsilon}$ and any pair $Q_1<Q_2$ with  $Q_1Q_2=2x^{4/7-\varepsilon}$, we can write\begin{align*}
		(1_{(\mathcal{M},2\mathcal{M}]}f_1)*\lambda^+(\cdot,h)=(1_{(\mathcal{M},2\mathcal{M}]}f_1)*(\lambda_1*\lambda_2)=((1_{(\mathcal{M},2\mathcal{M}]}f_1)*\lambda_1)*\lambda_2,
	\end{align*}
	where $\lambda_1$ and $\lambda_2$ have levels $Q_2/(2\mathcal{M})$ and $Q_1$ respectively, and satisfy $\lambda^+(\cdot,h)=\lambda_1*\lambda_2$. Thus  $(1_{(\mathcal{M},2\mathcal{M}]}f_1)*\lambda^+(\cdot,h)$ are well factorable of level $2x^{4/7-\varepsilon}$ for $2\mathcal{M}<x^{2/7-\varepsilon}$, and the same holds for $(1_{(\mathcal{M},2\mathcal{M}]}1_{2|\cdot}f_1)*\lambda^+(\cdot,h)$. By Lemma 2.5 and Lemma 2.6, the final remainder term is
	\begin{align*}
		\ll\frac{x}{(\log x)^{A-2}}.\tag{4.12}
	\end{align*}
	
	Now it remains to estimate the main term. Changing the order of summation, we get
	\begin{align*}
		\sum_{d'}\frac{\lambda_{d'}^+}{\varphi(md')}&\sum_{\substack{x^{\eta_2}/(\log x)^B<d\leq x^{\eta_1}\\2|dm,(d,md')=1}}\pi(x/d)=\sum_{\substack{x^{\eta_2}/(\log x)^B<d\leq x^{\eta_1}\\2|dm,(d,m)=1}}\pi(x/d)\sum_{(d',d)=1}\frac{\lambda_{d'}^+}{\varphi(md')}\\&=\sum_{\substack{x^{\eta_2}/(\log x)^B<d\leq x^{\eta_1}\\2|dm,(d,m)=1}}\pi(x/d)\sum_{d'|P'(z)}\frac{\lambda_{d'}^+}{\varphi(md')}\\&=\frac{F(2)+o(1)}{\varphi(m)}\sum_{\substack{x^{\eta_2}/(\log x)^B<d<x^{\eta_1}\\2|dm,(d,m)=1}}\frac{x}{d\log (x/d)}\prod_{\substack{(p,md)=1\\p\leq z}}\left(1-\frac{1}{p-1}\right)\prod_{\substack{p|m\\p\leq z}}\left(1-\frac{1}{p}\right)\\&=\frac{F(2)+o(1)}{m}H(m)\prod_{2<p\leq z}\left(\frac{p-2}{p-1}\right)\sum_{\substack{x^{\eta_2}/(\log x)^B<d<x^{\eta_1}\\2|dm,(d,m)=1}}\frac{x}{d\log (x/d)}H(d),\tag{4.13}
	\end{align*}
	where 
	\begin{align*}
		P'(z)=\prod_{\substack{p\leq z\\(p,d)=1}}p,\quad H(d)=\prod_{\substack{p|d\\p>2}}\left(\frac{p-1}{p-2}\right).
	\end{align*}
	The partial sum of $H(d)$ is well studied by La Bret\`{e}che, Pomerance and Tenenbaum [4] by
	employing the Selberg-Delange method: for any nonnegative integer $v$,
	\begin{align*}
		\sum_{\substack{d\leq y\\(d,a)=1}}H(2^vd)\sim \frac{A_0y}{2^{\varepsilon(a)}H(a)}\qquad(y\rightarrow\infty),
	\end{align*}
	where $\varepsilon(a)=1$ if $2|a$ and $\varepsilon(a)=0$ if $2\nmid a$, and $A_0$ is an absolute constant defined by
	\begin{align*}
		A_0:=\prod_{p>2}\left(1+\frac{1}{p(p-2)}\right).
	\end{align*}
	If $2|m$, we have
	\begin{align*}
		\sum_{\substack{x^{\eta_2}/(\log x)^B<d<x^{\eta_1}\\2|dm,(d,m)=1}}\frac{x}{d\log (x/d)}H(d)&=\sum_{\substack{x^{\eta_2}/(\log x)^B<d<x^{\eta_1}\\(d,m)=1}}\frac{x}{d\log (x/d)}H(d)\\&=(1+o(1))\frac{A_0}{2H(m)}\int_{x^{\eta_2}/(\log x)^B}^{x^{\eta_1}}\frac{x}{t\log(x/t)}\mathrm{d}t\\&=(1+o(1))\frac{A_0}{2H(m)}x\log \frac{1-\eta_2}{1-\eta_1}.
	\end{align*}
	If $2\nmid m$, we also have
	\begin{align*}
		\sum_{\substack{x^{\eta_2}/(\log x)^B<d<x^{\eta_1}\\2|dm,(d,m)=1}}\frac{x}{d\log (x/d)}H(d)&=\sum_{\substack{x^{\eta_2}/(2(\log x)^B)<d<x^{\eta_1}/2\\(d,m)=1}}\frac{x}{2d\log (x/(2d))}H(d)\\&=(1+o(1))\frac{A_0}{2H(m)}\int_{x^{\eta_2}/(2(\log x)^B)}^{x^{\eta_1}/2}\frac{x}{t\log (x/(2t))}\mathrm{d}t\\&=(1+o(1))\frac{A_0}{2H(m)}x\log \frac{1-\eta_2}{1-\eta_1}.
	\end{align*}
	By Mertens' theorem, we have
	\begin{align*}
		\prod_{2<p\leq z}\left(\frac{p-2}{p-1}\right)&=\prod_{2<p\leq z}\left(\frac{p-2}{p-1}\right)\left(1-\frac{1}{p}\right)^{-1}\prod_{2<p\leq z}\left(1-\frac{1}{p}\right)\\&=\frac{2e^{-\gamma}+o(1)}{\log z}\prod_{2<p\leq z}\frac{p(p-2)}{(p-1)^2}\\&=\frac{2e^{-\gamma}+o(1)}{\log z}\prod_{2<p\leq z}\left(1+\frac{1}{p(p-2)}\right)^{-1}\\&=\frac{2e^{-\gamma}+o(1)}{A_0\log z}.
	\end{align*}
	Then the RHS of (4.13) becomes
	\begin{align*}
		(1+o(1))\frac{2x}{m\log D}\log \frac{1-\eta_2}{1-\eta_1}.
	\end{align*}
	Summing over all integers $m$, we have
	\begin{align*}
		&\sum_{x^{1-\eta_1}<p'<x^{1-\eta_1+t_1+t_2}}\left(\log x-\frac{1}{1-\eta_1+t_1+t_2}\log p'\right)\sum_{x^{1-\eta_1}<p<x^{1-\eta_2}}\sum_{\substack{dp\leq x\\dp-1\equiv0\mkern-15mu\pmod{p'}}}1\\&\leq\left(\log\frac{1-\eta_2}{1-\eta_1}+o(1)\right)x\sum_{x^{\eta_1-t_1-t_2}/(\log x)^B<m<x^{\eta_1}}\left(\log x-\frac{\log (x/(m(\log x)^B))}{1-\eta_1+t_1+t_2}\right)\frac{2}{m\log D}\\&=\left(g_1(\eta_1,t_1,t_2)\log\frac{1-\eta_2}{1-\eta_1}+o(1)\right)x\log x\tag{4.14}
	\end{align*}
	where
	\begin{align*}
		g_1(\eta_1,t_1,t_2)=\left\{ \begin{aligned}
			&2\int_{\eta_1-t_1-t_2}^{\eta_1}\frac{1-(1-t)/(1-\eta_1+t_1+t_2)}{4/7-\varepsilon-t}\mathrm{d}t, \quad&& \eta_1<{2/7-\varepsilon};\\&2\int_{\eta_1-t_1-t_2}^{\eta_1}\frac{1-(1-t)/(1-\eta_1+t_1+t_2)}{1/2-t}\mathrm{d}t,\quad&& \text{otherwise}.
		\end{aligned}\right.
	\end{align*}
	By (4.9)-(4.14), we obtain a lower bound of $S_B$:
	\begin{align*}
		S_B\geq \left(\Bigg(\frac{1}{2}\left(\frac{1}{1-\eta_1+t_1}-\frac{1}{1-\eta_1+t_1+t_2}\right)-g_1(\eta_1,t_1,t_2)\Bigg)\log\frac{1-\eta_2}{1-\eta_1}+o(1)\right)x\log x.
	\end{align*}
	Then we need to find the largest $t_1\in(0,\eta_1)$ such that $S_B\geq 0$, i.e., for some $0<t_2<\eta_1-t_1$, one has
	\begin{align*}
		&\frac{1}{2}\left(\frac{1}{1-\eta_1+t_1}-\frac{1}{1-\eta_1+t_1+t_2}\right)>g_1(\eta_1,t_1,t_2).\tag{4.15}
	\end{align*}
	We now let \(t_1=t_1(\eta_1)\) be the function satisfying condition (4.15). (see Figure 2 below)
	By (4.8) and $S_B\geq 0$, we have
	\begin{align*}
		\widetilde{\mathscr{S}_{C_1}}&\geq x\left(\frac{1}{2(1-\eta_1+t_1(\eta_1))}\log \frac{1-\eta_2}{1-\eta_1}+\left(1-\frac{1}{1-\eta_1+t_1(\eta_1)}\right)\log \frac{1-\eta_2}{1-\eta_1}\right)+o(x)\\&=x\frac{1-2\eta_1+2t_1(\eta_1)}{2(1-\eta_1+t_1(\eta_1))}\log \frac{1-\eta_2}{1-\eta_1}+o(x).
	\end{align*}
	For $\nu=0.000001$, there exists $\Delta>0$ such that if $\max(\max_i|\eta_{i+1}-\eta_i|,1/2-\eta_1)<\Delta$, one has
	\begin{align*}
		\mathscr{S}'_3+\mathscr{S}'_4\geq \sum_{i=1}^N\widetilde{\mathscr{S}_{C_i}}+o(x)\geq& \left(\sum_{i=1}^{N}\frac{1-2\eta_i+2t_1(\eta_i)}{2(1-\eta_i+t_1(\eta_i))}\log\frac{1-\eta_{i+1}}{1-\eta_i}+o(1)\right)x\\\geq&\left(\int_c^{1/2}\frac{1-2\eta+2t_1(\eta)}{2(1-\eta)(1-\eta+t_1(\eta))}\mathrm{d}\eta-\nu+o(1)\right)x.
	\end{align*}
	For $\eta=\eta_1>0.1348$, by numerical calculation, we show the pragh of $t_1(\eta)$:
	\begin{figure}[ht]
		\centering
		\includegraphics[width=0.9\textwidth]{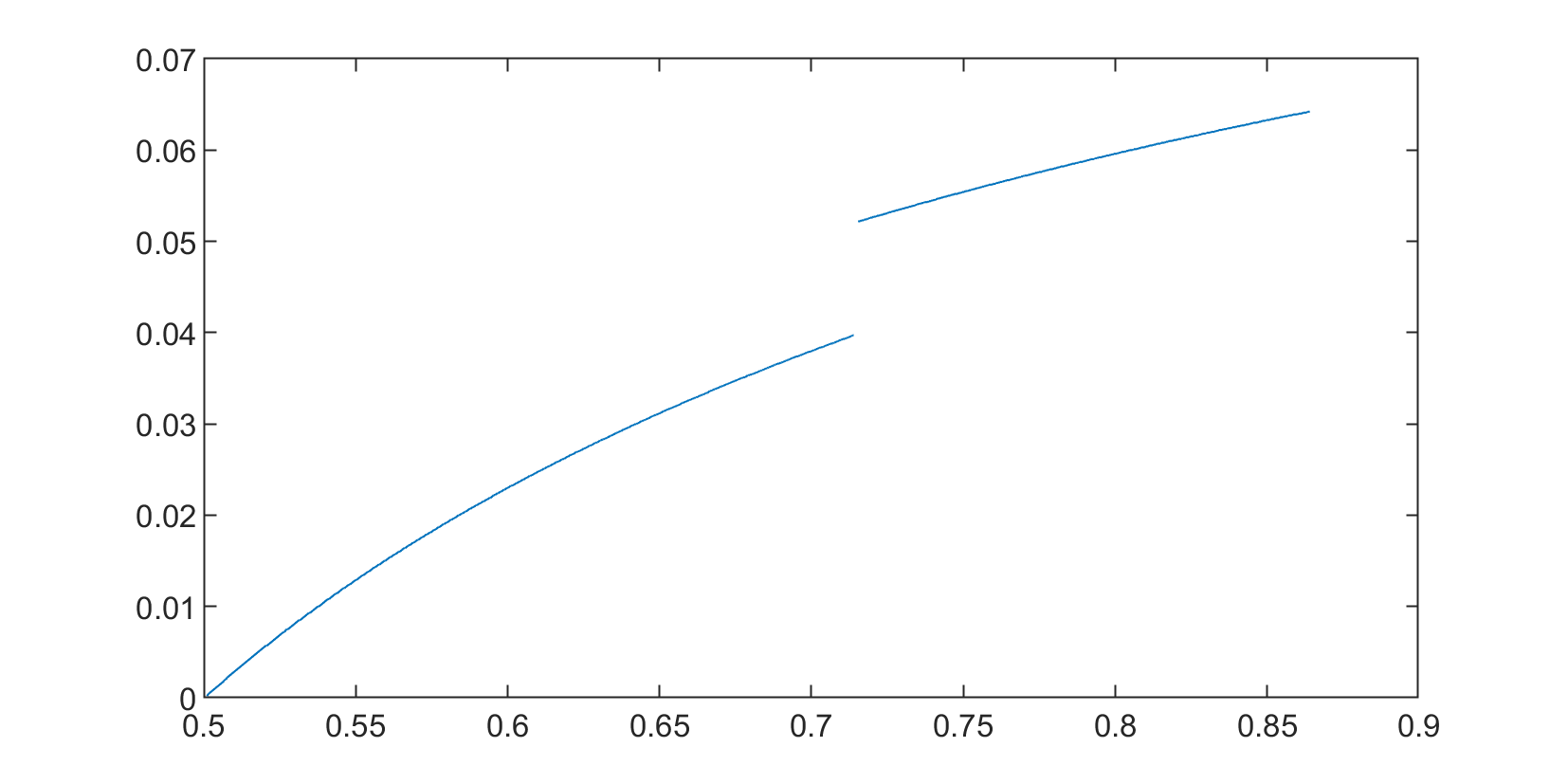} 
		\caption{$\eta\rightarrow t(\eta)=t_1(1-\eta)$}
	\end{figure}
	
	However, when $\eta_1\approx0.5$, the lower bound of $\widetilde{\mathscr{S}_{C_1}}$ above is not a good result. Note that\begin{align*}
		&\sum_{x^{1-\eta_i}<p<x^{1-\eta_{i+1}}}\sum_{dp\leq x}1-\sum_{x^{1-\eta_1}<p'<x}\sum_{x^{1-\eta_i}<p<x^{1-\eta_{i+1}}}\sum_{\substack{dp\leq x\\dp-1\equiv0\mkern-15mu\pmod{p'}}}1\\&\geq \sum_{x^{1-\eta_i}<p<x^{1-\eta_{i+1}}}\sum_{\substack{n\leq x-1,P^+(n)\leq x^{1/2}\\p|n+1}}1.
	\end{align*}
	After applying Theorem 3.2, we get
	\begin{align*}
		\mathscr{S}'_3+\mathscr{S}'_4\geq \left(\int_{c}^{1/2}\max\left(\frac{1-2\eta+2t_1(\eta)}{2(1-\eta)(1-\eta+t_1(\eta))},\frac{C(1-\eta)}{1-\eta}\right)\mathrm{d}\eta-\nu+o(1)\right)x.\tag{4.16}
	\end{align*}
	where the $C(\cdot)$ is the $C(\cdot)$ in Theorem 3.2.
	\subsection{Estimation of $\mathcal{R}$ and $\mathscr{S}'_5$}
	Recall\begin{align*}
		\mathcal{R}=\frac{1}{2}\sum_{x^{\delta_1}<p_1,p_2<x^{1/2}/(\log x)^B}\sum_{\substack{n\in S(x,x^{1/2}/(\log x)^B)\\p_1p_2|n+1}}1
	\end{align*}
	and\begin{align*}
		\mathscr{S}'_5=\sum_{x^{\delta_1}<p_1<x^{1/2}/(\log x)^B}\sum_{x^{1/2}<p_2<x^{1-\delta_1}}\sum_{\substack{n\in S(x,x^{1/2}/(\log x)^B)\\p_1p_2|n+1}}1.
	\end{align*}
	Thus for $(\alpha,\beta)\in\{(\delta_1,1/2-B\log\log x/\log x), (1/2,1-\delta_1)\}$, we consider
	\begin{align*}
		\mathop{\mathcal{R}}\limits^{\extsim[1]}(\alpha,\beta)&:=\sum_{\substack{x^{\delta_1}<p_1<x^{1/2}/(\log x)^B\\x^{\alpha}<p_2<x^{\beta}}}\sum_{\substack{n\in S(x,x^{1/2}/(\log x)^B)\\p_1p_2|n+1}}1.
	\end{align*}
	Note that\begin{align*}
		\mathop{\mathcal{R}}\limits^{\extsim[1]}(\alpha,\beta)&= \sum_{\substack{x^{\delta_1}<p_1<x^{1/2}/(\log x)^B\\x^{\alpha}<p_2<x^{\beta}}}\sum_{\substack{n\in S(x,x^{1/2}/(\log x)^B)\\p_1p_2|n+1}}\frac{\log n}{\log x}+o(x)\\&=\sum_{\substack{x^{\delta_1}<p_1<x^{1/2}/(\log x)^B\\x^{\alpha}<p_2<x^{\beta}}}\sum_{\substack{n<x\\p_1p_2|n+1}}\frac{\log h(n)}{\log x}-\sum_{\substack{x^{\delta_1}<p_1<x^{1/2}/(\log x)^B\\x^{\alpha}<p_2<x^{\beta}}}\sum_{\substack{n=dp<x\\x^{1/2}<p<x\\p_1p_2|n+1}}\frac{\log (x/p)}{\log x}+o(x),\tag{4.17}
	\end{align*}
	where  $h(n):=\sum_{p^k||n,p<x^{1/2}/(\log x)^B}p^k$.
	After invoking the identity
	\begin{align*}
		\log n=\sum_{m|n}\Lambda(m), 
	\end{align*}
	we obtain
	\begin{align*}
		\mathop{\mathcal{R}}\limits^{\extsim[1]}(\alpha,\beta)=&\frac{1}{\log x}\sum_{p'<x^{1/2}/(\log x)^B}\log p'\sum_{\substack{x^{\delta_1}<p_1<x^{1/2}/(\log x)^B\\x^{\alpha}<p_2<x^{\beta}}}\sum_{\substack{n<x\\p'|n\\p_1p_2|n+1}}1\\&-\sum_{\substack{x^{\delta_1}<p_1<x^{1/2}/(\log x)^B\\x^{\alpha}<p_2<x^{\beta}}}\sum_{\substack{n=dp<x\\x^{1/2}<p<x\\p_1p_2|n+1}}\frac{\log (x/p)}{\log x}+o(x)\\=&B_1-B_2+o(x).\tag{4.18}
	\end{align*}
	Let $n+1=d_1p_1p_2$. Omitting the details, applying Lemma 2.5 with $p\leftarrow p_1$, we obtain
	\begin{align*}
		B_1&=\frac{1}{\log x}\sum_{p'<x^{1/2}/(\log x)^B}\log p'\sum_{\substack{l=d_1p_2\\x^{\alpha}<p_2<x^{\beta}}}\sum_{\substack{lp_1\leq x\\x^{\delta_1}<p_1<x^{1/2}/(\log x)^B\\p'|lp_1-1}}1+o(x)\\&=\left(\frac{1}{2}g_2(\alpha,\beta)+o(1)\right)x,
	\end{align*}
	where\begin{align*}
		g_2(\alpha,\beta)=\left\{ \begin{aligned}
			&\left(\log\frac{1}{2\delta_1}\right)^2, \quad&& (\alpha,\beta)=(\delta_1,1/2-B\log\log x/\log x);\\&\int_{\delta_1}^{1/2}\frac{\log(2-2t)}{t}\mathrm{d}t,\quad&& (\alpha,\beta)=(1/2,1-\delta_1).
		\end{aligned}\right.
	\end{align*}
	For $B_2$, let $n+1=d_1p_1p_2$. After reducing the range of $d=d_1p_2$, we have
	\begin{align*}
		B_2&\geq \frac{1}{\log x}\sum_{x^{1/2}<p'<x}\log (x/p')\sum_{\substack{d=d_1p_2\\x^{1/2}(\log x)^{2B}<d<x^{1-\delta_1}/(\log x)^B\\x^{\alpha}<p_2<x^{\beta}}}\sum_{\substack{dp_1<x\\p'|dp_1-1}}1+o(x)=B_2'+o(x).
	\end{align*}
	Using the estimation of $B_1$ and the following fact that (the verification is left to interested readers.)
	\begin{align*}
		&\sum_{\substack{d=d_1p_2\\x^{1/2}(\log x)^{2B}<d<x^{1-\delta_1}/(\log x)^B\\x^{\alpha}<p_2<x^{\beta}}}\sum_{\substack{dp_1<x}}1=(g_2(\alpha,\beta)+o(1))x,\\&
		\sum_{\substack{d=d_1p_2\\x^{1/2}(\log x)^{2B}<d<x^{1-\delta_1}/(\log x)^B\\x^{\alpha}<p_2<x^{\beta}}}\sum_{\substack{dp_1<x}}1=\sum_{p'<x}\frac{\log p'}{\log x}\sum_{\substack{d=d_1p_2\\x^{1/2}(\log x)^{2B}<d<x^{1-\delta_1}/(\log x)^B\\x^{\alpha}<p_2<x^{\beta}}}\sum_{\substack{dp_1<x\\p'|dp_1-1}}1+o(x),\\&\sum_{x^{1/2}/(\log x)^B\leq p'\leq x^{1/2}}\frac{\log p'}{\log x}\sum_{\substack{d=d_1p_2\\x^{1/2}(\log x)^{2B}<d<x^{1-\delta_1}/(\log x)^B\\x^{\alpha}<p_2<x^{\beta}}}\sum_{\substack{dp_1<x\\p'|dp_1-1}}1=o(x),
	\end{align*}
	we get
	\begin{align*}
		\frac{1}{\log x}\sum_{x^{1/2}<p'<x}\log p'\sum_{\substack{d=d_1p_2\\x^{1/2}(\log x)^{2B}<d<x^{1-\delta_1}/(\log x)^B\\x^{\alpha}<p_2<x^{\beta}}}\sum_{\substack{dp_1<x\\p'|dp_1-1}}1=\left(\frac{1}{2}g_2(\alpha,\beta)+o(1)\right)x.
	\end{align*}
	Similar to (4.9) (but not the same), for $1/2<s<1$, we have
	\begin{align*}
		B_2'&=\frac{1}{\log x}\sum_{x^{1/2}<p'<x}\log (x/p')\sum_{\substack{d=d_1p_2\\x^{1/2}(\log x)^{2B}<d<x^{1-\delta_1}/(\log x)^B\\x^{\alpha}<p_2<x^{\beta}}}\sum_{\substack{dp_1<x\\p'|dp_1-1}}1\\&\geq \frac{1}{\log x}\left(\frac{1}{s}-1\right)\sum_{x^{1/2}<p'<x}\log p'\sum_{\substack{d=d_1p_2\\x^{1/2}(\log x)^{2B}<d<x^{1-\delta_1}/(\log x)^B\\x^{\alpha}<p_2<x^{\beta}}}\sum_{\substack{dp_1<x\\p'|dp_1-1}}1\\&\quad-\frac{1}{\log x}\sum_{x^{s}\leq p'<x}\Bigg(\frac{1}{s}\log p'-\log x\Bigg)\sum_{\substack{d=d_1p_2\\x^{1/2}(\log x)^{2B}<d<x^{1-\delta_1}/(\log x)^B\\x^{\alpha}<p_2<x^{\beta}}}\sum_{\substack{dp_1<x\\p'|dp_1-1}}1\\&=\left(\frac{1}{2}\left(\frac{1}{s}-1\right)g_2(\alpha,\beta)+o(1)\right)x-\mathcal{D}.\tag{4.19}
	\end{align*}
	where\begin{align*}
		\mathcal{D}=\frac{1}{\log x}\sum_{x^{s}\leq p'<x}\Bigg(\frac{1}{s}\log p'-\log x\Bigg)\sum_{\substack{d=d_1p_2\\x^{1/2}(\log x)^{2B}<d<x^{1-\delta_1}/(\log x)^B\\x^{\alpha}<p_2<x^{\beta}}}\sum_{\substack{dp_1<x\\p'|dp_1-1}}1.
	\end{align*}
	After applying the Rosser-Iwaniec sieve and using Lemma 2.5 to bound the remainder terms (We guess that Lemma 2.6 also holds for $l=d=d_1p_2$ with the above conditions, but this is not the key.), we can derive an upper bound for $\mathcal{D}$:
	\begin{align*}
		\mathcal{D}\leq 2x\int_{0}^{1-s}\frac{(1-t)/s-1}{1/2-t}\mathrm{d}tg_2(\alpha,\beta)+o(x).
	\end{align*}
	We thus have
	\begin{align*}
		B_2\geq B_2'+o(x)\geq x\left(\frac{1}{2}\left(\frac{1}{s}-1\right)-2\int_{0}^{1-s}\frac{(1-t)/s-1}{1/2-t}\mathrm{d}t\right)g_2(\alpha,\beta)+o(x).
	\end{align*}
	Taking $s=0.882$, by numerical calculation, we have
	\begin{align*}
		&B_2\geq (0.0
		3249g_2(\alpha,\beta)+o(1))x.
	\end{align*}
	Then (4.18) becomes \begin{align*}
		&\mathop{\mathcal{R}}\limits^{\extsim[1]}(\alpha,\beta)\leq (0.46751g_2(\alpha,\beta)+o(1))x,
	\end{align*}
	and we have\begin{align*}
		\mathcal{R}\leq\left(0.233755\Bigg(\log\frac{1}{2\delta_1}\Bigg)^2+o(1)\right)x,\quad
		\mathscr{S}'_5\leq \left(0.46751\int_{\delta_1}^{1/2}\frac{\log{(2-2t)}}{t}\mathrm{d}t+o(1)\right)x.\tag{4.20}
	\end{align*}
	\begin{remark}
		In [27, \S4.2], Wang used a parameter $\theta\in (0,1/2-\varepsilon)$ and splitted $P^+(n)\leq x^{1/2-\varepsilon}$ into $x^\theta<P^+(n)\leq x^{1/2-\varepsilon}$ and $P^+(n)\leq x^{\theta}$, which should be compared with (4.17) and (4.18). (Note that we have enlarged $\mathcal{R}$ and $\mathscr{S}'_5$.)
	\end{remark}
	
	\subsection{Conclusion}
	We conclude from (4.1), (4.2), (4.3), (4.16), (4.20) and Lemma 4.1 that
	\begin{align*}
		\sum_{\substack{n<x\\P^+(n)<P^+(n+1)}}1\geq (\mathcal{C}(c,\delta_1)-2\nu)x+o(x),
	\end{align*} 
	where $\nu=0.000001$, and for the $C(\cdot)$ in Theorem 3.2, $\mathcal{C}(c,\delta_1)$ is defined by
	\begin{align*}
		\mathcal{C}(c,\delta_1):=&\log\frac{1}{1-c}-2\int_0^c\log\left(\frac{1-t}{1-c}\right)\frac{\mathrm{d}t}{4/7-t}+\int_{\delta_1}^{1/2}\rho\left(\frac{1}{t}\right)\frac{1}{t}\mathrm{d}t\\&+\int_c^{1/2}\max\left(\frac{1-2\eta+2t_1(\eta)}{2(1-\eta)(1-\eta+t_1(\eta))},\frac{C(1-\eta)}{1-\eta}\right)\mathrm{d}\eta\\&-0.233755\Bigg(\log\frac{1}{2\delta_1}\Bigg)^2-0.46751\int_{\delta_1}^{1/2}\frac{\log{(2-2t)}}{t}\mathrm{d}t
	\end{align*}
	with $0<c<2/7-\varepsilon$ and $1/3<\delta_1<1/2$. By numerical calculation, we have
	\begin{align*}
		\mathcal{C} (c,\delta_1)-2\nu> 0.280,
	\end{align*}
	by taking $c=0.1348$, $\delta_1=0.417$. 
	We complete the proof of Theorem 1.4.

	\section{Proof of Theorem 1.5}
	We consider that $n$ is a $y$-smooth number and $n+1$ has two prime factors $p_1$ and $p_2$ satisfying \begin{align*}
		\frac{33}{107}-2\varepsilon<\frac{\log p_1}{\log x}<\frac{33}{107}-\varepsilon<\frac{\log p_2}{\log x}<\frac{33}{107}.
	\end{align*}
	Then we have
	\begin{align*}
		P^+(n)<y<P^+(n+1)\leq x^{41/107+3\varepsilon},
	\end{align*}
	where $y=x^{1/C}$ with the $C$ in Lemma 2.3.
	Thus\begin{align*}
		\sum_{\substack{n<x\\P^+(n)<P^+(n+1)\leq x^{41/107+3\varepsilon}}}1&\geq \frac{1}{2}\sum_{n\in S(x,y)}\sum_{\substack{p_1p_2|n+1\\x^{33/107-2\varepsilon}<p_1<x^{33/107-\varepsilon}<p_2<x^{33/107}}}1\\&=\frac{1}{2}\sum_{x^{33/107-2\varepsilon}<p_1<x^{33/107-\varepsilon}<p_2<x^{33/107}}\sum_{\substack{n\in S(x,y)\\n\equiv-1\mkern-15mu\pmod{p_1p_2}}}1\\&=\frac{1}{2}(\mathscr{C}_1+\mathscr{C}_2),\tag{5.1}
	\end{align*}
	where
	\begin{align*}
		&\mathscr{C}_1=\sum_{x^{33/107-2\varepsilon}<p_1<x^{33/107-\varepsilon}<p_2<x^{33/107}}\frac{1}{\varphi(p_1p_2)}\sum_{\substack{n\in S(x,y)\\(n,p_1p_2)=1}}1,\\&\mathscr{C}_2=\sum_{x^{33/107-2\varepsilon}<p_1<x^{33/107-\varepsilon}<p_2<x^{33/107}}\Bigg(\sum_{\substack{n\in S(x,y)\\n\equiv-1\mkern-15mu\pmod{p_1p_2}}}1-\frac{1}{\varphi(p_1p_2)}\sum_{\substack{n\in S(x,y)\\(n,p_1p_2)=1}}1\Bigg).
	\end{align*}
	By Lemma 2.3, we have
	\begin{align*}
		\mathscr{C}_2&\ll\sum_{q\leq x^{66/107-\varepsilon}}\Bigg|\sum_{\substack{n\in S(x,y)\\n\equiv-1\mkern-15mu\pmod{p_1p_2}}}1-\frac{1}{\varphi(q)}\sum_{\substack{n\in S(x,y)\\(n,q)=1}}1\Bigg|\\&\ll\frac{x}{(\log x)^A}.\tag{5.2}
	\end{align*}
	By Lemma 2.1 and Mertens' theorem, we have
	\begin{align*}
		\mathscr{C}_1=x\rho\left(C\right)\log\frac{33/107}{33/107-\varepsilon}\log\frac{33/107-\varepsilon}{33/107-2\varepsilon}\left(1+O\left(\frac{1}{\log x}\right)\right).\tag{5.3}
	\end{align*}
	Conlcuding (5.1), (5.2) and (5.3), we complete the proof of Theorem 1.5.
	
	\section{Proof of Theorem 1.6}
	Applying the following lemma, we begin to estimate $T_c(x)$.\begin{lemma}
		For $0<c<1$ and sufficiently large $x$, we have
		\begin{align*}
			T_c(x)=\sum_{\substack{p< x\\P^+(p-1)> x^{c}}}1+O\left(\frac{x\log\log x}{(\log x)^2}\right).
		\end{align*}
	\end{lemma}\begin{proof}
		This is [31, Theorem 2].
	\end{proof}
	Using the following fact that
	\begin{align*}
		&\sum_{p<x}1=\frac{x}{\log x}+o\left(\frac{x}{\log x}\right),\\
		&\sum_{p< x}1=\frac{1}{\log x}\sum_{p'<x}\log p'\sum_{\substack{p<x\\p'|p-1}}1+o\left(\frac{x}{\log x}\right),\\&\frac{1}{\log x}\sum_{p'<x^{1/2}/(\log x)^B}\log p'\sum_{\substack{p<x\\p'|p-1}}1=\frac{1}{2}\frac{x}{\log x}+o\left(\frac{x}{\log x}\right),
	\end{align*}
	we get
	\begin{align*}
		\frac{1}{\log x}\sum_{x^{1/2}/(\log x)^B\leq p'<x}\log p'\sum_{\substack{p<x\\p'|p-1}}1&=\frac{1}{2}\frac{x}{\log x}+o\left(\frac{x}{\log x}\right).\tag{6.1}
	\end{align*}
	For $1/2<c<1$ and $0<t_1<1-c$, by (6.1), we have
	\begin{align*}
		\sum_{\substack{p< x\\P^+(p-1)> x^{c}}}1=&\sum_{x^c<p'<x}\sum_{\substack{p<x\\p'|p-1}}1= \frac{1}{c+t_1}\frac{1}{\log x}\sum_{x^{1/2}/(\log x)^B\leq p'<x}\log p'\sum_{\substack{p<x\\p'|p-1}}1\\&+\frac{1}{\log x}\sum_{x^c<p'<x}\left(\log x-\frac{1}{c+t_1}\log p'\right)\sum_{\substack{p<x\\p'|p-1}}1\\&-\frac{1}{c+t_1}\frac{1}{\log x}\sum_{x^{1/2}/(\log x)^B\leq p'\leq x^c}\log p'\sum_{\substack{p<x\\p'|p-1}}1\\=&\frac{1}{2(c+t_1)}\frac{x}{\log x}+o\left(\frac{x}{\log x}\right)-\mathcal{D},\tag{6.2}
	\end{align*}
	where
	\begin{align*}
		\mathcal{D}=&\frac{1}{\log x}\sum_{x^c<p'<x}\left(\frac{1}{c+t_1}\log p'-\log x\right)\sum_{\substack{p<x\\p'|p-1}}1\\&+\frac{1}{c+t_1}\frac{1}{\log x}\sum_{x^{1/2}/(\log x)^B\leq p'\leq x^c}\log p'\sum_{\substack{p<x\\p'|p-1}}1.
	\end{align*}
	We want to find the largest $t_1$ such that $\mathcal{D}\geq 0$.
	Similar to (4.9), for $t_2\in(0,1-c-t_1)$, we have
	\begin{align*}
		\mathcal{D}\geq&\left(\frac{1}{c+t_1}-\frac{1}{c+t_1+t_2}\right)\frac{1}{\log x}\sum_{x^{1/2}/(\log x)^B\leq p'<x}\log p'\sum_{\substack{p<x\\p'|p-1}}1\\&-\frac{1}{\log x}\sum_{x^{c}<p'<x^{c+t_1+t_2}}\left(\log x-\frac{1}{c+t_1+t_2}\log p'\right)\sum_{\substack{p<x\\p'|p-1}}1\\=&\frac{1}{2}\left(\frac{1}{c+t_1}-\frac{1}{c+t_1+t_2}\right)\frac{x}{\log x}+o\left(\frac{x}{\log x}\right)-\mathcal{E},
	\end{align*}
	where
	\begin{align*}
		\mathcal{E}=\frac{1}{\log x}\sum_{x^{c}<p'<x^{c+t_1+t_2}}\left(\log x-\frac{1}{c+t_1+t_2}\log p'\right)\sum_{\substack{p<x\\p'|p-1}}1. 
	\end{align*}
	Then we only need to estimate $\mathcal{E}$. For $c=u_1<u_2<\cdots<u_N=c+t_1+t_2$, we have
	\begin{align*}
		\mathcal{E}&\leq \sum_{i}\frac{1}{\log x}\left(\log x-\frac{1}{c+t_1+t_2}\log x^{u_i}\right)\sum_{x^{u_i}<p'<x^{u_{i+1}}}\sum_{\substack{p<x\\p'|p-1}}1\\&\leq\sum_{i}\left(1-\frac{u_i}{c+t_1+t_2}\right)\sum_{\substack{x^{1-u_{i+1}}/(\log x)^B<l<x^{1-u_i}\\2|l}}\sum_{\substack{lp\leq x\\lp+1 \ \text{is prime}}}1+o\left(\frac{x}{\log x}\right).\tag{6.3}
	\end{align*}
	As similar arguments to that of [7, Theorem 1.1], we can get
	\begin{align*}
		\sum_{\substack{x^{1-u_{i+1}}/(\log x)^B<l<x^{1-u_i}\\2|l}}\sum_{\substack{lp\leq x\\lp+1 \ \text{is prime}}}1\leq \frac{7}{2}\log\frac{u_{i+1}}{u_i}\frac{x}{\log x}+o\left(\frac{x}{\log x}\right)
	\end{align*}
	Then there exists $\Delta>0$ such that if $\max_{i}|u_i-u_{i+1}|<\Delta$, we have\begin{align*}
		\mathcal{E}&\leq \frac{7}{2}\sum_{i}\left(1-\frac{u_i}{c+t_1+t_2}\right)\log\frac{u_{i+1}}{u_i}\frac{x}{\log x}+o\left(\frac{x}{\log x}\right)\\&\leq \left(\frac{7}{2}\int_c^{c+t_1+t_2}\left(1-\frac{u}{c+t_1+t_2}\right)\frac{\mathrm{d}u}{u}+\nu\right)\frac{x}{\log x}+o\left(\frac{x}{\log x}\right)\\&=\left(\frac{7}{2}\log\frac{c+t_1+t_2}{c}-\frac{7}{2}\frac{t_1+t_2}{c+t_1+t_2}+\nu\right)\frac{x}{\log x}+o\left(\frac{x}{\log x}\right),
	\end{align*}
	where $\nu=0.000001$.
	We need to find the largest $t_1\in(0,1-c)$ such that $\mathcal{D}\geq 0$, i.e., for some $0<t_2<1-c-t_1$, one has 
	\begin{align*}
		\frac{1}{2}\left(\frac{1}{c+t_1}-\frac{1}{c+t_1+t_2}\right)>\frac{7}{2}\log\frac{c+t_1+t_2}{c}-\frac{7}{2}\frac{t_1+t_2}{c+t_1+t_2}+\nu.\tag{6.4}
	\end{align*}
	Now let $t_1=t_1(c)$ be the function satistying the condition (6.4). By numerical calculation, we present the graph of $t_1(c)$:
	
	\begin{figure}[ht]
		\centering
		\includegraphics[width=0.9\textwidth]{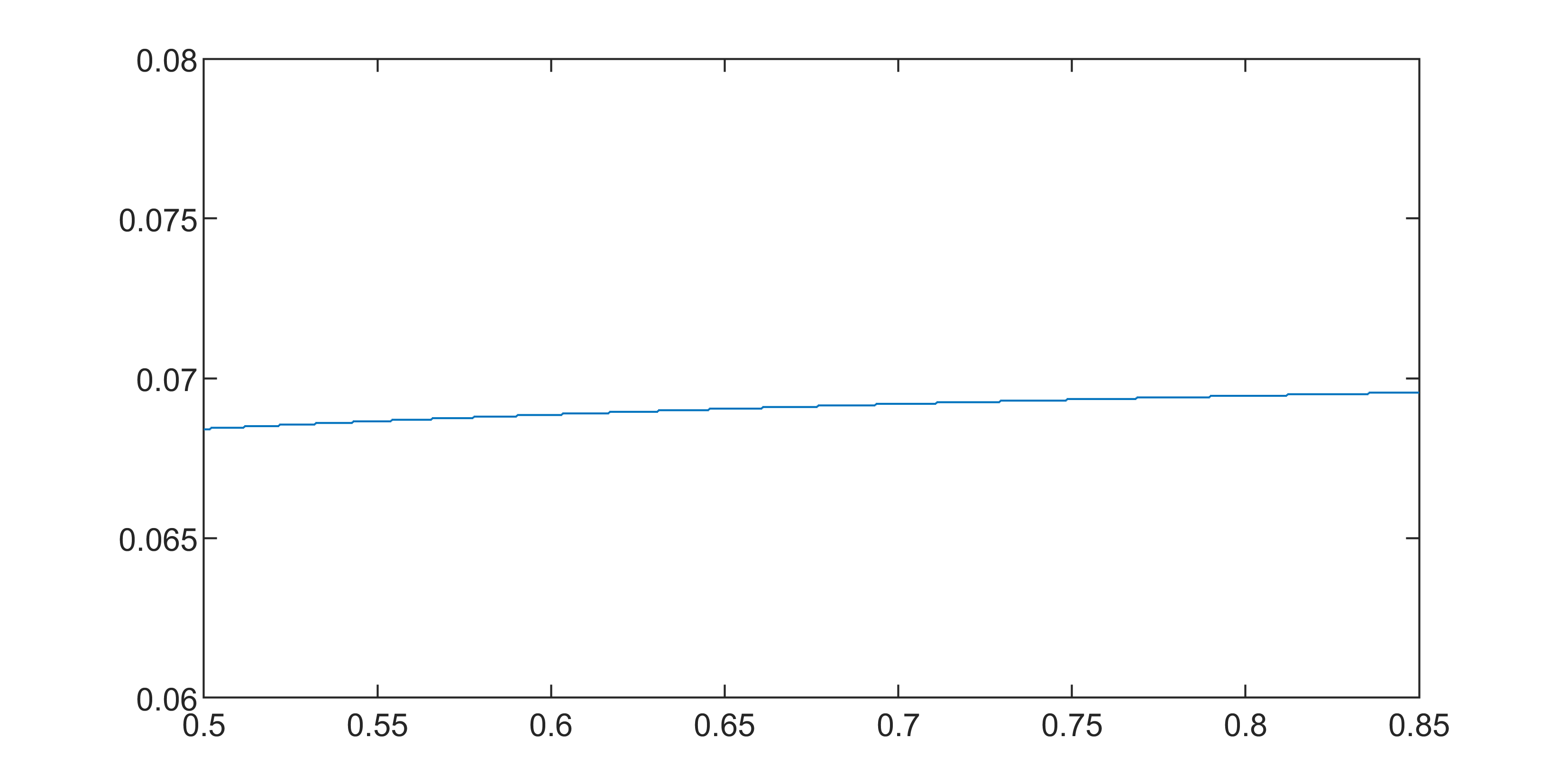} 
		\caption{$c\rightarrow t_1(c)$}
	\end{figure}
	
	By Lemma 6.1, (6.2) and $\mathcal{D}\geq 0$, for $1/2<c<1$, we have\begin{align*}
		T_c(x)\leq \frac{1}{2(c+t_1)}\frac{x}{\log x}+o\left(\frac{x}{\log x}\right).
	\end{align*}
	We also write
	\begin{align*}
		T_c(x)\leq \frac{1-\delta}{2c}\frac{x}{\log x}+o\left(\frac{x}{\log x}\right),
	\end{align*}
	where $1/2<c<1$ and $\delta=\delta(c)=t_1/(c+t_1)>0$. We complete the proof of Theorem 1.6.

	\setcounter{theorem}{1}
	
	\begin{remark}
		For (6.3), we use Ding and Wang's results [7] for convenience. In fact, we can achieve better results by considering the following question: What function $C(\eta)$ satisfies the inequality $\#\{p<x:p\equiv 1\mkern-7mu\pmod{q}\}\leq (C(\eta)+\varepsilon)x/(\varphi(q)\log x)$ for almost all $q\sim x^{\eta}$, while the number of exceptions $q$ is less than \(x^\eta (\log x)^{-A}\)? This question has been studied by many scholars. For example, we can see Fouvry's work [11, Theorem 3]. 
	\end{remark}

	\noindent\textbf{Acknowledgements}\ \ The author thanks Professor Zhiwei Wang for his valuable advice and constant encouragement.


\begin{thebibliography}{99}
		
		
		
		\bibitem{}E. Bombieri, J. B. Friedlander and H. Iwaniec, \textit{Primes in arithmetic progressions to large moduli}, Acta Math., \textbf{156} (1986), no.~3-4, 203--251.
		
		\bibitem{}F. Chen and Y. Chen, \textit{On the largest prime factor of shifted primes}, Acta Math. Sin. (Engl. Ser.), \textbf{33} (2017), no.~3, 377-382.
		
		\bibitem{}J. M. De Koninck and N. Doyon, \textit{On the distance between smooth numbers}, Integers, \textbf{11} (2011), 647--669.
		
		\bibitem{}R. de la Bretèche, C. Pomerance and G. Tenenbaum, \textit{Products of ratios of consecutive integers}, Ramanujan J., \textbf{9} (2005), no.~1-2, 131--138.
		
		\bibitem{}Y. Ding, \textit{On a conjecture on shifted primes with large prime factors}, Arch. Math. (Basel), \textbf{120} (2023), no.~3, 245--252.
		
		\bibitem{}Y. Ding, \textit{On a conjecture on shifted primes with large prime factors, II}, Bull. Aust. Math. Soc., \textbf{111} (2025), no.~1, 48--55.
		
		\bibitem{}Y. Ding and Z. Wang, \textit{Note on shifted primes with large prime factors}. arXiv:2510.04026, preprint (2025).
		
		\bibitem{}P. Erd\H{o}s, \textit{On the normal number of prime factors of $p-1$ and some related problems concerning Euler’s $\phi$-function}, Q. J. Math., Oxf. Ser., \textbf{6} (1935), 205--213.
		
		\bibitem{}P. Erd\H{o}s, \textit{Some unconventional problems in number theory}, Astérisque, \textbf{61} (1979), 73--82.
		
		\bibitem{}P. Erd\H{o}s and C. Pomerance, \textit{On the largest prime factors of $n$ and $n+1$}, Aequationes Math., \textbf{17} (1978), no.~2-3, 311--321.
		
		\bibitem{}E. Fouvry, \textit{Théorème de Brun-Titchmarsh: application au théorème de Fermat} (in French), Invent. Math., \textbf{79} (1985), no.~2, 383--407.
		
		\bibitem{Fouvry}E. Fouvry, G. Tenenbaum, \textit{Entiers sans grand facteur premier en progressions arithmetiques}, Proc. London Math. Soc. (3), \textbf {63} (1991), no.~3, 449--494.
		
		\bibitem{Hildebrand}A. Hildebrand, \textit{On the number of positive integers $\leq x$ and free of prime factors $>y$}, J. Number Theory, \textbf {22} (1986), no.~3, 289--307.
		
		\bibitem{Iwaniec}H. Iwaniec, \textit{A new form of the error term in the linear sieve}, Acta Arith., \textbf {37} (1980), 307--320.
		
		\bibitem{}Y. Jiang, G. Lü and Z. Wang, \textit{Averaged forms of two conjectures of Erd\H{o}s and Pomerance, and their applications}, Adv. Math., \textbf{409} (2022), part A, Paper No. 108592, 42 pp.
		
		\bibitem{Lü}X. Lü, Z. Wang and B. Chen, \textit{On the smooth values of shifted almost-primes}, Int. J. Number Theory, \textbf{15} (2019), no.~1, 1--9.
		
		\bibitem{}X. Lü and Z. Wang, \textit{On the largest prime factors of consecutive integers}, Monatsh. Math., \textbf{206} (2025), no.~2, 403--418.
		
		\bibitem{}C.D. Pan, X. Ding and Y. Wang, \textit{On the representation of every large even integer as a sum of a prime and an almost prime}, Sci. Sinica, \textbf{18} (1975), no.~5, 599--610.
		
		\bibitem{}A. Pascadi, \textit{Smooth numbers in arithmetic progressions to large moduli},
		Compos. Math., \textbf{161} (2025), no.~8, 1923--1974.
		
		\bibitem{Pascadi}A. Pascadi, \textit{On the exponents of distribution of primes and smooth numbers}. arXiv:2505.00653, preprint (2025).
		
		\bibitem{}J. Rivat, \textit{On pseudo-random properties of $P(n)$ and $P(n+1)$}, Period. Math. Hungar., \textbf{43} (2001), no.~1-2, 121--136.
		
		\bibitem{}V.T. Sós, \textit{
			Turbulent years: Erd\H{o}s in his correspondence with Turán from 1934 to 1940}, in: Paul Erdős and His Mathematics, I (Budapest, 1999), Bolyai Soc. Math. Stud. 11 (János Bolyai Math. Soc., Budapest, 2002), 85--146.
		
		\bibitem{}T. Tao and J. Ter\"av\"ainen, \textit{The structure of correlations of multiplicative functions at almost all scales, with applications to the Chowla and Elliott conjectures}, Algebra Number Theory, \textbf{13} (2019), no.~9, 2103--2150.
		
		\bibitem{}G. Tenenbaum, \textit{Some of Erd\H{o}s' unconventional problems in number theory, thirty-four years later}, in: Erd\H{o}s Centennial, Bolyai Soc. Math. Stud. 25, János Bolyai Math Soc., Budapest, 2013, 651–681.  
		
		\bibitem{}J. Ter\"av\"ainen, \textit{On binary correlations of multiplicative functions}, Forum Math. Sigma, \textbf{6} (2018),  Paper No. e10, 41 pp.
		
		\bibitem{}Z. Wang, \textit{On the largest prime factors of consecutive integers in short intervals}, Proc. Amer. Math. Soc., \textbf{145} (2017), no.~8, 3211--3220.
		
		\bibitem{}Z. Wang, \textit{Autour des plus grands facteurs premiers d'entiers consécutifs voisins d'un entier criblé}, Q. J. Math., \textbf{69} (2018), no.~3, 995--1013. 
		
		\bibitem{Wang}Z. Wang, \textit{Sur les plus grands facteurs premiers d'entiers consécutifs}, Mathematika, \textbf{64} (2018), no.~2, 343--379.
		
		\bibitem{}Z. Wang, \textit{Three conjectures on $P^+(n)$ and $P^+(n+1)$ hold under the Elliott-Halberstam conjecture for friable integers}, J. Number Theory, \textbf{223} (2021), 1--11.
		
		\bibitem{}D. Wolke, \textit{Über die mittlere Verteilung der Werte zahlentheoretischer Funktionen auf Restklassen. II} (in German),
		Math. Ann., \textbf{204} (1973), 145--153.
		
		\bibitem{}J. Wu, \textit{On shifted primes with large prime factors and their products}, Arch. Math. (Basel), \textbf{112} (2019), no.~4, 387--393.
		
		
		
		
		
		
		
		
		
	\end{thebibliography}
\end{document}